\documentclass[a4paper,10pt,leqno]{amsart}
\usepackage{amsmath}
\usepackage{amsthm}
\usepackage{amssymb}
\usepackage{stmaryrd}
\usepackage{url}
\usepackage{here}
\usepackage[all]{xy}
\usepackage{comment}
\usepackage{mathtools}
\usepackage{ascmac}
\usepackage{mathrsfs}
\usepackage{bm}
\usepackage{ulem}
\usepackage{framed}
\usepackage{graphicx}
\usepackage{lmodern}
\usepackage{caption}
\usepackage{subcaption}
\usepackage{footnote}
\usepackage{amscd}
\usepackage{mathdots}
\usepackage{subcaption}
\usepackage{float}
\mathtoolsset{showonlyrefs}
\usepackage{comment}
\usepackage{booktabs}
\usepackage{multirow}

\usepackage{xcolor}
\definecolor{webgreen}{rgb}{0,.5,0}
\definecolor{webbrown}{rgb}{.6,0,0}
\definecolor{RoyalBlue}{cmyk}{1, 0.50, 0, 0}
\usepackage[colorlinks=true, breaklinks=true, urlcolor=webbrown, linkcolor=RoyalBlue, citecolor=webgreen, backref=page]{hyperref}

\calclayout
\numberwithin{equation}{section}

\newtheorem{theorem}{Theorem}[section]
\newtheorem{lemma}[theorem]{Lemma}
\newtheorem{proposition}[theorem]{Proposition}
\newtheorem{corollary}[theorem]{Corollary}
\newtheorem{definition}[theorem]{Definition}
\theoremstyle{remark}
\newtheorem{remark}{Remark}

\newcommand{\dd}{\mathrm{d}}

\newcommand{\R}{\mathbb{R}}
\newcommand{\C}{\mathbb{C}}
\newcommand{\Z}{\mathbb{Z}}
\newcommand{\N}{\mathbb{N}}
\newcommand{\TT}{\mathbb{T}}

\newcommand{\ip}[2]{\langle #1, #2 \rangle}

\newcommand{\red}[1]{{\color{red} #1}}
\newcommand{\blue}[1]{{\color{blue} #1}}

\newcommand{\violet}[1]{{\color{violet} #1}}
\newcommand{\hide}[1]{}

\newcommand{\stkout}[1]{\ifmmode\text{\sout{\ensuremath{#1}}}\else\sout{#1}\fi}

\makeatletter
\def\l@subsection{\@tocline{2}{0pt}{2.5pc}{5pc}{}}
\makeatother

\title{Reduction of Multiple Orthogonal Polynomials to Standard Orthogonal Polynomials}

\author[Gharakhloo]{Roozbeh Gharakhloo}
\address{(R. Gharakhloo) Mathematics Department, University of California, Santa Cruz, CA 95064, USA}
\email{roozbeh@ucsc.edu}

\author[Kosmakov]{Maksim Kosmakov}
\address{(M. Kosmakov) Department of Mathematical Sciences, University of Cincinnati, P.O. Box 210025, Cincinnati, OH 45221, USA}
\email{kosmakmm@ucmail.uc.edu}

\author[Miyahara]{Kenta Miyahara}
\address{(K. Miyahara) Department of Mathematical Sciences, 
Indiana University Indianapolis,
402 N. Blackford St., Indianapolis, IN 46202 USA}
\email{kemiya@iu.edu}

\begin{document}

\begin{abstract}
In this article, we derive explicit formulae expressing multiple orthogonal polynomials in terms of standard orthogonal polynomials. We treat both the real-line and unit-circle settings: multiple orthogonal polynomials on the real line (MOPRL) are reduced to orthogonal polynomials on the real line (OPRL), while multiple orthogonal polynomials on the unit circle (MOPUC) are reduced to orthogonal polynomials on the unit circle (OPUC). These formulae also yield corresponding reductions of the Christoffel--Darboux kernels, from the MOPRL kernel to the OPRL kernel and from the MOPUC kernel to the OPUC kernel. In particular, they recover Zinn-Justin's kernel for the external-source random matrix model \cite{ZinnJustin1997} and Baik's kernel reduction formula in the spiked source model \cite{Bai09}. We also apply our general results to concrete examples: in the real-line setting, we obtain an explicit expression for the multiple Hermite Christoffel--Darboux kernel in terms of classical Hermite polynomials, while in the unit-circle setting, we use arc-indicator weights to exhibit
resonance-type zero escape phenomena for type II MOPUC.
\vspace{.05cm}
\end{abstract}

%\date{\today}
\maketitle
%\tableofcontents

\setlength{\parskip}{6pt}

%%%%%%%%%%%%%%%%%%%%%%%%%%%%%%%%%%%%%%%%%%%%

\section{Introduction}

Multiple orthogonal polynomials are a natural extension of ordinary orthogonal
polynomials in which a single orthogonality measure is replaced by a finite
collection of measures.  They were originally developed in connection with
Hermite-Padé approximation, analytic number theory, and the theory of special
functions; see, for example, \cite{Apt98,Kui10,VA20} and the references therein.
Given weights \(w^{(1)},\ldots,w^{(r)}\) and a multi-index
\(\vec n=(n_1,\ldots,n_r)\), the type II multiple orthogonal polynomial is the
monic polynomial of degree \(|\vec n|\) satisfying \(n_j\) orthogonality
conditions with respect to the weight \(w^{(j)}\).  The dual type I object is a
 linear combination of the weights \(w^{(1)},\ldots,w^{(r)}\), with
polynomial coefficients, satisfying corresponding moment conditions.  When
\(r=1\), these objects reduce to the usual orthogonal polynomials.

In this paper we derive explicit reduction formulae for multiple orthogonal polynomials on the real line and on the unit circle when $r \geq 2$. These formulae express the type II polynomials, the type I linear forms, and the associated Christoffel--Darboux kernels in terms of usual orthogonal polynomials. As discussed in Section~\ref{section baik}, a closely related reduction at the level of Christoffel--Darboux kernels was obtained by Baik \cite{Bai09} in connection with the finite-rank external-source random matrix model. From this perspective, the present work extends Baik's finite-rank reduction formula from the special external-source multi-index $(n-r,1,\ldots,1)$ to general perfect systems of multiple orthogonality on the real line. It also develops the parallel unit-circle theory. A further feature of our approach is that these reductions are obtained by finite-dimensional linear algebra, without using Riemann--Hilbert machinery.

More precisely, we choose one of the weights \(w^{(1)},\ldots,w^{(r)}\),
and, after relabeling if necessary, denote it by \(w^{(1)}\).  We regard this
weight as the \textit{reference}
%\footnote{\blue{I would suggest we use \textit{reference} weight instead of primary weight. "primary" implies importance and suggest the other weights are secondary in importance.}} 
weight and use the corresponding ordinary orthogonal polynomials as a basis. 
In this basis, the mixed-moment matrix for the multiple orthogonality conditions has a
block upper-triangular structure.  The lower-right block, denoted \(M_{22}\),
contains precisely the information coming from the companion weights
%\footnote{\blue{I would suggest we use \textit{companion} weights instead of perturbing weights; In know this is the case in most examples we have in mind, but regarding the general theory, the remaining weights are not necessarily ``perturbations'' of the reference weight. ``Supplementary'' is still better than ``perturbing'' but still implies some secondary status. So I think ``companion'' is the best.}}
\(w^{(2)},\ldots,w^{(r)}\).  Our main results show that the type II multiple
orthogonal polynomial, the type I linear form, and the corresponding
Christoffel--Darboux kernel can all be written \textit{explicitly} in terms of the
ordinary orthogonal polynomials for \(w^{(1)}\), the ordinary
Christoffel-Darboux kernel, and the inverse of \(M_{22}\). In the finite-rank external-source setting, this recovers Baik's kernel
reduction formula.  The same formalism also recovers Zinn-Justin's earlier
external-source kernel; see
Section~\ref{section baik}.

We also prove analogues of these reduction formulae for multiple orthogonal
polynomials on the unit circle.  Multiple orthogonal polynomials on the unit
circle were introduced by Mínguez Ceniceros and Van Assche \cite{MCVA08}, who
related them to rational approximation of Carathéodory functions,
formulated Riemann--Hilbert problems for the type I and type II objects, and
used this formulation to derive recurrence relations for MOPUC.  Further
structural results, including normality and recurrence relations obtained
without relying on the Riemann--Hilbert formulation, were developed in
\cite{CBDMO15}.  More recently, Kozhan and Vaktnäs established Szegő-type
recurrences for MOPUC \cite{KV24a}.  Our MOPUC
formulae reduce type I and type II MOPUC, and their Christoffel--Darboux kernel,
to ordinary OPUC data plus a finite-sized mixed-moment matrix.  This gives a direct unit-circle counterpart of the real-line reduction.

The results of this paper are intended as a first step in a broader research program. In forthcoming work, we plan to apply these reduction formulae to several problems, including some unexplored aspects of external-source GUE and CUE models; ensembles of non-intersecting random walks on $\mathbb Z$, in both discrete and continuous time, with consecutive initial points and multiple packets of consecutive end points \cite{GL}; and ensembles of non-intersecting Brownian bridges on the circle \cite{BuckinghamLiechty17} with one starting point and multiple ending points.

\subsection{Outline} We now describe the organization of the paper.  Section~\ref{section: prelim} reviews the
preliminaries for MOPRL and MOPUC.  We define type I and type II multiple
orthogonal polynomials, introduce monotone paths of multi-indices, recall the
associated biorthogonality relations, and set up the mixed-moment matrices whose
block structure is the main linear-algebraic input of the paper. In Section~\ref{section main results}
we present the main results.  For MOPRL, Theorems~\ref{typeIIreduction}--\ref{kernel reduction} express the
type II polynomial, the type I linear form, and the Christoffel--Darboux kernel
in terms of ordinary OPRL data. Corollary~\ref{cor-general-Hermite-kernel} specializes the real-line kernel reduction to the multiple Hermite weights for a general multi-index.  Theorems~\ref{thm:typeII}--\ref{thm:kernel} give the corresponding
statements for MOPUC; Corollaries~\ref{cor: MOPUC type II eg 1}--\ref{cor: MOPUC kernel eg 2}
specialize these results to arc-indicator weights and, in particular, exhibit
zero-escape phenomena for type II MOPUC.

Sections~\ref{section proofs} and~\ref{appendix mopuc} contain the proofs of the MOPRL and MOPUC reduction theorems,
respectively.  The proofs are based only on finite-dimensional linear algebra,
the block structure of the mixed-moment matrix, and the biorthogonality of the
type I and type II systems.  Section~\ref{section example MOPUC} contains the proofs of Corollaries~\ref{cor: MOPUC type II eg 1}--\ref{cor: MOPUC kernel eg 2} mentioned above.  Section~\ref{section baik} explains how Baik's reduction formula for
the Christoffel--Darboux kernel of the finite-rank external-source ensemble is
recovered from our general MOPRL theorem.  In the same section, we also provide a derivation of a formula
of Zinn-Justin \cite{ZinnJustin1997} for the external-source ensemble kernel using our formalism.  Appendix~\ref{Mixed-moment and hankel} relates the
mixed-moment matrix formalism to bordered Hankel determinants and to the
standard striped Hankel determinant formula for type II multiple orthogonal
polynomials.

\subsection*{Acknowledgements.} 
This work was initiated during the workshop \textit{``Riemann-Hilbert problems, Toeplitz matrices, and applications''} hosted by the American Institute of Mathematics (AIM) in March 2024. 
The authors are grateful to the AIM and the workshop organizers.
The third author was partially supported by a scholarship from the Japan Student Services Organization and the
Hokushin Scholarship Foundation.
The authors also thank Pavel Bleher, Percy Deift, Vitaly Tarasov, Walter Van Assche and Maxim Yattselev for helpful discussions and remarks. 
%\sout{The authors are also grateful to Professor V. Tarasov for a helpful suggestion that substantially simplified the proof of Theorem~\ref{kernel reduction}.}

\section{Preliminaries} \label{section: prelim}

\subsection{MOPRL preliminaries}

Let $\vec{n} = (n_1, \cdots, n_r) \in \N^r$ be a multi-index of size $n = |\vec{n}| = \sum_{j=1}^r n_j$, and let $\{ \mu_j \}_{j=1}^r$ be a sequence of positive measures on the real line such that all moments $\int_\R x^k \dd \mu_j(x)$ exist.
Furthermore, we assume each $\mu_j$ is absolutely continuous with respect to the Lebesgue measure. Thus, we have a sequence of weight functions $\{w^{(j)}(x)\}_{j=1}^r$ such that
\[
\dd \mu_j = w^{(j)}(x) \dd x, \quad j = 1, 2, \cdots, r.
\]
In this setting, we define two types of multiple orthogonal polynomials on the real line.

\begin{definition} \label{defn of type II}
A type II MOPRL $P_{\vec{n}}(x)$ is a monic polynomial of degree $n$ satisfying
\begin{equation} \label{type II orthogonality}
	\int_{\R} P_{\vec{n}}(x) x^k w^{(j)}(x) \dd x = 0, \quad k=0, \cdots, n_j-1, \quad j=1, \cdots, r.
\end{equation}
\end{definition}

\begin{definition} \label{defn of type I}
Type I MOPRL are a nonzero vector of polynomials $(A_1(x), \cdots, A_r(x))$ where each $A_j(x)$ has degree at most $n_j-1$ and the linear form 
\[
Q_{\vec{n}}(x) = \sum_{j=1}^r A_j(x) w^{(j)}(x)
\]
satisfies
\begin{equation} \label{type I orthogonality 1}
\int_{\R} Q_{\vec{n}}(x) x^k \dd x = 
\begin{dcases}
    0, & k = 0, \cdots, n-2,\\
    1, & k = n-1.
\end{dcases}
\end{equation}
\end{definition}

Condition \eqref{type I orthogonality 1} gives a linear system of $n$ equations for the $n$ unknown coefficients of $A_j(x)$, $j=1, \cdots, r$.
Thus, the existence and uniqueness of type I MOPRL depend on the non-singularity of this linear system. 
In other words, if the moment matrix 
\[
m = \begin{pmatrix}
    m^{(1)} & m^{(2)} & \cdots & m^{(r)}
\end{pmatrix}
\]
where each $m^{(j)}$ is an $n \times n_j$ Hankel matrix whose $(a,b)$-entry is given by $\int_\R x^{a + b - 2} w^{(j)}(x) \dd x$, has a nonzero determinant, then type I MOPRL uniquely exists. This condition also guarantees the existence and uniqueness of type II MOPRL.

\begin{definition}
Given weights $\{w^{(j)}(x)\}_{j=1}^r$, if for every multi-index $\vec{n}\in\mathbb{N}^r$, the associated type I and II MOPRL uniquely exist, $\{w^{(j)}(x)\}_{j=1}^r$ is called a perfect system.
\end{definition}

Throughout this article, we assume that the underlying system of weights is perfect, so that MOPRL exists uniquely for every multi-index. This condition is satisfied, for example, by Angelesco systems, AT systems, and Nikishin systems, see for example \cite{VA20}. Also, in this article, we refer to \(w^{(1)}(x)\) as the reference weight and to the remaining weights as the companion weights.
For the reference weight $w^{(1)}(x)$, we write the corresponding monic orthogonal polynomials by $\{p_k(x)\}_{k \geq 0}$, where they satisfy
\begin{equation}\label{orthogonality w1}
	\int_\R p_n(x)p_k(x) w^{(1)}(x) \dd x = h_n \delta_{nk}, \quad k,n \in \Z_{\geq 0}.
\end{equation}
We refer them to the standard orthogonal polynomials on the real line.
Also, we define the Christoffel-Darboux kernel for this OPRL by
\begin{align} \label{defn of CD kernel for p}
K_0(x,y)=\sum_{j=0}^{n-1}\frac{p_j(x)p_j(y)}{h_j}w^{(1)}(y).
\end{align}

%%%%%%%%%%%%%%%%%%%%%%%%%%%%%%%%%%%%%%

A \textit{monotonic path} is a sequence of multi-indices $\{\vec{n}_j\}_{j=0}^n \subset \N^r$ from $\vec{n}_0 = \vec{0}$ (all zero vector) to $\vec{n}_n = \vec{n}$ such that $\vec{n}_{j+1}$ is obtained from $\vec{n}_j$ by increasing exactly one component by $1$. 
In other words, for each $j=0, 1, \cdots, n-1$, the path must satisfy
\begin{equation}
	|\vec{n}_j| = j, \quad \text{and} \quad \vec{n}_{j+1} \geq \vec{n}_j,
\end{equation}
where the inequality is taken component-wise. 
A \textit{natural path} is an example of the monotonic path, where the components of the multi-index are filled in sequential order. 
That is, the path first increments the first component from $0$ to $n_1$, then increments the second component from $0$ to $n_2$, and so on, until the final index $\vec{n}_n = (n_1, \cdots, n_r)$ is reached. In particular, the natural path is unique for a fixed multi-index $\vec{n}$.

For a fixed monotonic path $\{\vec{n}_j\}_{j=0}^n \subset \N^r$, we define the associated sequences of type II MOPRL and type I linear forms by
\begin{equation} \label{defn of MOPRL systems}
	P_j(x) := P_{\vec{n}_j}(x) \quad \text{and} \quad Q_j(x) := Q_{\vec{n}_{j+1}}(x).
\end{equation}
These sequences are the building blocks for the Christoffel-Darboux kernel for MOPRL,
\begin{equation}\label{defn of CD kernel}
	K(x,y) = \sum_{j=0}^{n-1} P_j(x)Q_j(y).
\end{equation}
A crucial result, shown by Daems and Kuijlaars (see \cite[Theorem 1.1, Remark 1.2]{DaemsKuijlaars2004}), is that this sum is independent of the specific monotonic path chosen to connect $\vec{0}$ and $\vec{n}_n$.
This path independence allows us to select the most convenient monotonic path, i.e., the natural path.
Therefore, whenever we mention $P_j(x)$ and $Q_j(x)$ defined in \eqref{defn of MOPRL systems} throughout this paper, we assume that the multi-index $\vec{n}$ is fixed and the sequences $P_j(x)$ and $Q_j(x)$ are understood to be defined along the natural path from $\vec 0$ to $\vec n$.

%That is, the path first increments the component corresponding to the weight $w^{(1)}(x)$ from $0$ to $n_1$, then increments the component for $w^{(2)}(x)$ from $0$ to $n_2$, and so on, until the final index $\vec{n}_n = (n_1, \cdots, n_r)$ is reached. 

%One can readily show that sequences of multiple orthogonal polynomials satisfy the bi-orthogonality, see \cite[Lemma 3.1]{DaemsKuijlaars2004} and \cite[Theorem 23.1.6]{Ismail}, and a simple degeneration if the given multi-index stops at $\vec{n}_n=(n,0,\cdots,0)$. For later use, we formulate them as the following lemmas.

Another important fact about \eqref{defn of MOPRL systems} is that they satisfy the bi-orthogonality, see \cite[Lemma 3.1]{DaemsKuijlaars2004} and \cite[Theorem 23.1.6]{Ismail}. For later use, we formulate it as the following lemma.

\begin{lemma} \label{biorthogonality}
%For a monotonic path of multi-indices, the associated type II MOPRL and type I linear forms satisfy the bi-orthogonality relation:
One has
\begin{equation}\label{biorth MOPRL}
	\int_{\R} P_k(x) Q_m(x) \dd x = \delta_{km}, \quad \forall k,m \in \mathbb{Z}_{\geq 0}.
\end{equation}
\end{lemma}

\begin{comment}
\begin{lemma} \label{degeneration}
For the special case $\vec{n}_n =(n,0,\cdots,0)$, one has
\begin{equation}
     P_{k}(x) = p_k(x),\quad Q_{k}(x)= \frac{p_{k}(x) }{h_{k}}w^{(1)} (x), \quad k = 0, \cdots, n.
\end{equation}
\end{lemma}
\end{comment}

%%%%%%%%%%%%%%%%%%%%%%%%%%%%%%%%%%%%%%%%%%%%%%

%\subsection{Key Objects and Their Block Structure.}

Before formulating our main results, we introduce some notations. 
First, we denote the space of polynomials on the real line with real coefficients of degree at most $n$ by $\mathcal{P}_n^{\R} = \mathrm{span}\{1, x, \cdots, x^n\}$.
For us, the most interesting basis is the set of orthogonal polynomials $\{p_k(x)\}_{k=0}^n$ associated with $w^{(1)}$.
Let $\bm{p}(x)$ be the $n$-column vector of \textit{basis polynomials},
\begin{equation} \label{defn of pi}
    \bm{p}(x):=(p_0(x),\cdots,p_{n-1}(x))^\top,
\end{equation}
$\bm{q}(x)$ be the $n$-column vector of \textit{weighted basis polynomials},
\begin{multline} \label{defn of phi}
    \bm{q}(x) := \bigg( p_0(x) w^{(1)}(x), \cdots, p_{n_1-1}(x) w^{(1)}(x),\\ p_0(x) w^{(2)}(x), \cdots, p_{n_2-1}(x) w^{(2)}(x), \boldsymbol{\cdots}, p_0 (x) w^{(r)}(x), \cdots, p_{n_r-1}(x) w^{(r)}(x)\bigg)^\top,
\end{multline}
and $\bm{v}$ be the $n$-column vector of \textit{mixed-moments of $p_n(x)$},
\begin{equation} \label{defn of v}
    \bm{v} := \int_{\R} p_n(x) \bm{q}(x) dx.
\end{equation}
By \eqref{orthogonality w1}, one can immediately observe that the first $n_1$ components of $\bm{v}$ are zero.

Next, we construct the $n \times n$ \textit{mixed-moment matrix} $\bm{M}$ associated with the basis $p_m(x)$ and the weighted family $p_k (x) w^{(j)}(x)$,
\begin{equation} \label{defn of M}
\bm{M} = \int_{\R} \bm{p}(s) \,\bm{q}(s)^\top \, \dd s.
\end{equation}
For $k = 1, \cdots, r$, if we write a $n \times n_k$ matrix $M^{(k)}$ whose $(a,b)$-entry is given by
\[
M^{(k)}_{a, b} = \int_{\R} p_{a-1}(x) p_{b-1}(x) w^{(k)}(x) \dd x,
\]
one can readily see that $\bm{M}$ is of the form,
\begin{align} \label{M=(M1 M2 ... Mr)}
\bm{M} = \begin{pmatrix}
    M^{(1)} & M^{(2)} & \cdots & M^{(r)}
\end{pmatrix}.
\end{align}
We shall also define a $n \times n$ matrix of the norms of $p_k(x)$,
\begin{equation} \label{defn of H}
    \bm{H} = \operatorname{diag}(h_0,h_1,\ldots, h_{n-1}).
\end{equation}

For later use, it is convenient to partition these objects according to the split
between the reference system, consisting of the first \(n_1\) components, and the
companion system, consisting of the remaining \(n-n_1\) components. We use the
block decompositions
\begin{equation} \label{defn of pi phi v, 1-2}
\bm{p}(x) =
\begin{pmatrix}
\bm{p}_1(x) \\
\bm{p}_2(x)
\end{pmatrix},
\qquad
\bm{q}(x) =
\begin{pmatrix}
\bm{q}_1(x) \\
\bm{q}_2(x)
\end{pmatrix},
\qquad
\bm{v} =
\begin{pmatrix}
\bm{0} \\
\bm{v}_2
\end{pmatrix},
\end{equation}
and
\begin{equation} \label{defn of M, H 4 blocks}
\bm{M} =
\begin{pmatrix}
\bm{M}_{11} & \bm{M}_{12} \\
\bm{M}_{21} & \bm{M}_{22}
\end{pmatrix},
\qquad
\bm{H} = \operatorname{diag}(\bm{H}_1,\bm{H}_2).
\end{equation}

We note that the top-left $n_1 \times n_1$ block matrix $\bm{M}_{11}$ is, in fact, $\bm{H}_1$, and the bottom-left $(n-n_1) \times n_1$ block matrix $\bm{M}_{21}$ is the zero matrix $\bm{0}$; by the orthogonality condition \eqref{orthogonality w1}. The right blocks, $\bm{M}_{12}$ and $\bm{M}_{22}$, involve integrals against the companion weights $w^{(j)}(x)$ for $2 \leq j \leq r$, and cannot be simplified immediately. Thus, $\bm{M}$ has the block upper-triangular structure,
\begin{equation} \label{M block structure}
    \bm{M}  = \begin{pmatrix}\bm{H}_1 & \bm{M}_{12}  \\ \bm{0} & \bm{M}_{22}  \end{pmatrix}.
\end{equation}

\begin{remark}
    Observe that using this notation the  Christoffel-Darboux kernel for  OPRL \eqref{defn of CD kernel for p}  can be written as 
     \begin{equation}\label{K0 block}
    K_0(x,y)
    =\bm q_1(y)^\top\bm H_1^{-1}\bm p_1(x)
     +w^{(1)}(y)\bm p_2(y)^\top\bm H_2^{-1}\bm p_2(x).
\end{equation}
\end{remark}

%%%%%%%%%%%%%%%%%%%%%%%%%%%%%%%%%%%%%%%%%%%%%%

\subsection{MOPUC preliminaries}
%\footnote{\red{Subsection titles could be improved, I would say, let us combine subsections 2.1, 2.2, and 2.3 into one subsection and call it ``MOPRL preliminaries'', and call the subsection on MOPUC as ``MOPUC preliminaries''}}

As we present the results for multiple orthogonal polynomials on the unit circle (MOPUC) later, we shall give a similar preparation now for MOPUC as we did for MOPRL in the above three subsections. 
We start by reviewing a few facts about the orthogonal polynomials on the unit circle (OPUC).
Let $\nu_1$ be a positive measure on the unit circle $\TT$ such that all moments exist. With an inner product,
\begin{equation}
\ip{f(z)}{g(z)} = \int_{\TT} f(z) \overline{g(z)} \, \dd \nu_1,
\end{equation}
$L^2(\nu_1)$ is a Hilbert space.
Let $\mathcal{P}^{\TT}_n = \mathrm{span} \{1, z, \cdots, z^n\}$ be the space of polynomials with complex coefficients and of degree at most $n$. Clearly, $\mathcal{P}^{\TT}_n$ is a subspace of $L^2(\nu_1)$. 
Applying the Gram-Schmidt process to the standard basis $\{z^n\}_{n \geq 0}$ of $\mathcal{P}^{\TT}_n$, one obtains the \textit{monic} orthogonal basis $\{ \varphi_n(z) \}_{n \geq 0}$ that satisfy
\begin{equation} \label{opuc orthogonality}
\int_{\TT} \varphi_n(z) \, \overline{\varphi_m(z)} \, \dd \nu_1 = h_n \delta_{nm}, \quad h_n > 0,
\end{equation}
which are called the OPUC.
Assuming $\nu_1$ is absolutely continuous with respect to the normalized arclength measure, i.e.,
\[
\dd \nu_1 (z) = \omega^{(1)}(z) \frac{|\dd z|}{2\pi},
\]
with a nonnegative weight function $\omega^{(1)}(z)$ defined on $\TT$, we define the Christoffel-Darboux kernel for OPUC by
\begin{equation}\label{eq:Szego-kernel}
\mathcal K_0(z,\xi) = \sum_{j=0}^{n-1} \frac{\varphi_j(z) \, \overline{\varphi_j(\xi)}}{h_j} \, \omega^{(1)}(\xi).
\end{equation}

As with the generalization of OPRL to MOPRL, we consider the multiple orthogonal polynomials on the unit circle (MOPUC) following \cite{MCVA08, CBDMO15, KV24a}, and derive reduction formulae of MOPUC to OPUC.
To this end, we introduce a few notations and definitions here.
Let $\vec{n} = (n_1, n_2, \cdots, n_r) \in \N^r$ be a multi-index of size $n=|\vec{n}|=\sum_{j=1}^r n_j$.
Define measures $\nu_j$ on $\TT$ for $j = 1, \cdots, r$ such that all moments exist, and we assume each $\nu_j$ is absolutely continuous with respect to the normalized arclength measure,
\begin{align} \label{nuj and wj}
\dd \nu_j(z) = \omega^{(j)}(z) \frac{|\dd z|}{2\pi}
\end{align}
where each $\omega^{(j)}(z)$ is a nonnegative\footnote{Following \cite{MCVA08} and the subsequent studies on MOPUC \cite{CBDMO15, KV24a, MOPUC4, MOPUC5, MCVA08}, we restrict our attention in this work to nonnegative measures, although the theory can also be developed for complex measures.} weight function defined on $\TT$.
Now, we are ready to define the type I and II MOPUC. 

\begin{definition}\label{def:typeII MOPUC}
A type~II MOPUC $\Phi_{\vec{n}}(z)$ is a monic polynomial of degree $n$ such that
\begin{equation} \label{eq:typeII MOPUC-orth}
\int_{\TT} \Phi_{\vec{n}}(z) \, \overline{z^k} \, \omega^{(j)}(z) \frac{|\dd z|}{2\pi} = 0,
\quad k=0,1,\cdots,n_j-1, \quad j=1,2,\cdots,r.
\end{equation}
\end{definition}

\begin{definition}\label{def:typeI MOPUC}
Type~I MOPUC are a nonzero vector of polynomials $(\Lambda_{1}, \cdots, \Lambda_{r})$ with $\deg \Lambda_{j} \leq n_j - 1$ 
\begin{comment}
\red{that satisfies
\begin{align}
    \sum_{j = 1}^{r} \ip{\Lambda_{\vec{n},j}}{z^k}_j = \delta_{k,n-1}, \quad k = 0, 1, \cdots, n-1.
\end{align}
}
\end{comment}
such that the linear form
\begin{equation} \label{eq:typeIMOPC-linearform}
\Psi_{\vec{n}}(z) = \sum_{j=1}^r \Lambda_{j}(z) \,  \omega^{(j)}(z),
\end{equation}
satisfies
\begin{align}
\int_{\TT} \Psi_{\vec{n}}(z) \, \overline{z^k} \, \frac{|\dd z|}{2\pi} = \begin{dcases}
    0, & k = 0, \cdots, n-2,\\
    1, & k = n-1.
\end{dcases}\label{eq:typeI MOPC-orth}
\end{align}
\end{definition}

\begin{definition} \label{def:normalperfect}
Given weights $\{\omega^{(j)}(z)\}_{j=1}^r$, if $\Phi_{\vec{n}}(z)$ and $\Psi_{\vec{n}}(z)$ uniquely exist for every multi-index $\vec{n} \in \N^r$, then the system is called perfect.
\end{definition}

\begin{remark} \label{rem 1}
Replacing $\{z^k\}_{0\le k\le n_j-1}$ by any invertible lower-triangular change of basis (for example, $\{\varphi_k(z)\}_{0 \leq k\le n_j-1}$) in \eqref{eq:typeII MOPUC-orth} and \eqref{eq:typeI MOPC-orth} leaves the linear systems defining $\Phi_{\vec{n}}(z)$ and $\Psi_{\vec{n}}(z)$ unchanged.
\end{remark}

A monotonic path is a sequence of multi-indices $\{\vec{n}_j\}_{j=0}^n$ from $\vec{n}_0=\vec{0}$ to $\vec{n}_n = \vec{n}$ such that $|\vec{n}_j| = j$ and $\vec{n}_{j+1} \geq \vec{n}_j$.
Along a fixed path, we define the associated sequences of type II polynomials and type I linear forms as
\begin{equation}\label{eq:path-PQ}
\Phi_j(z) := \Phi_{\vec{n}_j}(z), \quad \Psi_j(z) := \Psi_{\vec{n}_{j+1}}(z), \quad j=0,1,\dots,n-1.
\end{equation}
We then define the Christoffel-Darboux kernel for MOPUC by
\begin{equation}\label{eq:Kn-def}
\mathcal K(z,\xi)=\sum_{j=0}^{n-1} \Phi_j(z) \, \overline{\Psi_j(\xi)}.
\end{equation}
\begin{remark}
This Christoffel-Darboux kernel is slightly different from the one considered in \cite[Theorem 6.1]{KV24a}. Kozhan-Vaktn\"as did not assume the absolute continuity for each $\nu_j$ with respect to the arclength measure, so it was natural for them to deal with the vector of type I polynomials $\Lambda_{\vec{n}} = (\Lambda_{1}, \cdots, \Lambda_{r})$, which results in to have the Christoffel-Darboux kernel as a vector form. However, Kozhan-Vaktn\"as's and ours are essentially the same, and we understand the kernel defined in \eqref{eq:Kn-def}  as an analogue of \eqref{defn of CD kernel}, where we already assumed the absolute continuity of the measures on the real line with respect to Lebesgue measure.
\end{remark}

Similarly to \cite[Theorem 6.1]{KV24a}, one can readily show that $\mathcal K(z, \xi)$ is independent of the specific monotonic path chosen to connect $\vec{0}$ and $\vec{n}_n$. Thus, we use the natural path hereafter.
For this reason, whenever we mention $\Phi_j(z)$ and $\Psi_j(z)$ defined in \eqref{eq:path-PQ} for the rest of the paper, we assume that the multi-index $\vec{n}$ is fixed and the sequences $\Phi_j(x)$ and $\Psi_j(x)$ are understood to be defined along the natural path from $\vec 0$ to $\vec n$. 

As with Lemma~\ref{biorthogonality}, a similar claim holds for MOPUC sequences, see \cite[Lemma 4.1]{KV24a}.

\begin{lemma} \label{lem: orth + degeneration}
The bi-orthogonality condition,
\begin{equation}\label{eq:biorth}
\int_{\TT} \Phi_k(z) \, \overline{\Psi_{m}(z)} \, \frac{|\dd z|}{2\pi} = \delta_{km},
\end{equation}
\begin{comment}
\[
\red{\sum_{j = 1}^r \ip{\Phi_{\vec{n}_k}}{\Lambda_{\vec{n}_{m+1}, j}}_j = \delta_{km}},
\]
\end{comment}
holds for any $k,m \in \Z_{\geq 0}$. 
\begin{comment}
For the special case $\vec{n}_n =(n,0,\cdots,0)$, MOPUC reduces to standard OPUC,
\begin{equation}
     \Phi_{k}(x) = \varphi_k(x),\quad \Psi_{k}(x)= \frac{\varphi_{k}(x)}{h_{k}} \omega^{(1)}(x), \quad k = 0, \cdots, n.
\end{equation} 
\end{comment}
\end{lemma}

%%%%%%%%%%%%%%%%%%%%%%%%%%%%%%%%%%%%%%%%%%%%%%

As a final preparation for Theorems~\ref{thm:typeII}, \ref{thm:typeI}, and \ref{thm:kernel}, we introduce some notations here.
Let $\bm{\varphi}(z)$ be the $n$-column vector of standard basis polynomials,
\begin{equation}\label{eq:pi-primary}
\boldsymbol{\varphi}(z) = \big( \varphi_0(z),\varphi_1(z),\cdots,\varphi_{n-1}(z) \big)^\top,
\end{equation}
$\boldsymbol{\psi}(z)$ be the $n$-column vector of weighted basis polynomials,
\begin{multline} \label{eq:chi-stacked}
    \boldsymbol{\psi}(z) := \bigg( \varphi_0(z) \omega^{(1)}(z), \cdots, \varphi_{n_1-1}(z) \omega^{(1)}(z),\\ \varphi_0(z) \omega^{(2)}(z), \cdots, \varphi_{n_2-1}(z) \omega^{(2)}(z), \boldsymbol{\cdots}, \varphi_0 (z) \omega^{(r)}(z), \cdots, \varphi_{n_r-1}(z) \omega^{(r)}(z)\bigg)^\top,
\end{multline}
and $\boldsymbol{u}$ be the $n$-column vector of mixed-moments of $\varphi_n(z)$,
\begin{equation}\label{eq:v-def}
\boldsymbol{u} = \int_{\TT} \varphi_n(z) \,\overline{\boldsymbol{\psi}(z)}\, \frac{|\dd z|}{2\pi},
\end{equation}
where one can readily observe that the first $n_1$ components are zero by \eqref{opuc orthogonality}.
Then, we define the $n \times n$ mixed-moment matrix as with the MOPRL case%\footnote{\red{It is better to use $\boldsymbol{\mathcal{M}}$ here to be consistent with $\boldsymbol{M}$ for OPRL.} \textcolor{purple}{$\to$ Edited as requested.}},
\begin{equation}\label{eq:M-def}
\bm{\mathcal{M}} = \int_{\TT} \bm{\varphi}(z) \, \overline{\bm \psi(z)}^\top \, \frac{|\dd z|}{2 \pi} = \begin{pmatrix}
    \mathcal{M}^{(1)} & \mathcal{M}^{(2)} & \cdots & \mathcal{M}^{(r)}
\end{pmatrix},
\end{equation}
where
\[
\mathcal{M}^{(k)}_{a, b} = \int_{\TT} \varphi_{a-1}(z) \, \overline{\varphi_{b-1}(z)} \, \omega^{(k)}(z) \frac{|\dd z|}{2\pi}.
\]
It is convenient to partition these objects according to the split between the reference system (first $n_1$ components) and the companion systems (remaining $n-n_1$ components),
\begin{equation}\label{eq:block-structure}
\bm{\mathcal{M}} = \begin{pmatrix}
\bm{\mathcal{M}}_{11} & \bm{\mathcal{M}}_{12}\\
\bm{\mathcal{M}}_{21} & \bm{\mathcal{M}}_{22}
\end{pmatrix}, \quad
\bm{\varphi}(z) = \begin{pmatrix} 
\bm{\varphi}_1(z)\\ 
\bm{\varphi}_2(z)
\end{pmatrix}, \quad
\bm{\psi}(z) = \begin{pmatrix}
\bm{\psi}_1(z)\\ 
\bm{\psi}_2(z)
\end{pmatrix}, \quad
\bm{u} = \begin{pmatrix}
\bm{0}\\
\bm{u}_2
\end{pmatrix}.
\end{equation}
Using \eqref{opuc orthogonality} and by the very definition of $\bm{\mathcal{M}}$, one can readily show that $\bm{\mathcal{M}}$ has a block upper-triangular structure
\begin{equation}\label{eq:M-upper-tri}
\bm{\mathcal{M}} = \begin{pmatrix}
\bm{\mathcal H}_1 & \bm{\mathcal{M}}_{12}\\
\bm{0} & \bm{\mathcal{M}}_{22}
\end{pmatrix}, \quad 
\bm{\mathcal{H}} = \operatorname{diag}(h_0, \cdots, h_{n-1}) = \operatorname{diag}(\bm{\mathcal H}_1, \bm{\mathcal H}_2).
\end{equation}
Note that each entry of $\mathcal{H}$ is positive. 

\section{Main Results} \label{section main results}

\subsection{MOPRL}

Our main results provide explicit formulas that express MOPRL and their associated kernel $K(x,y)$ in terms of the standard OPRL $\{p_k(x)\}_{k \geq 0}$ and its kernel $K_0(x,y)$. We present proofs of these theorems in Section~\ref{section proofs}.

\begin{theorem}\label{typeIIreduction}
The type II MOPRL $P_{\vec{n}}(x)$
%\footnote{\magenta{there seems to be inconsistent notation we use both $P_{\vec{n}}(x)$ and $P_n(x)$ (see MOPRL preliminaries). I prefer $P_{\vec{n}}(x)$ as it records $\vec{n}$ not just $n = |\vec{n}|$ which is a vague notation (we can have $\vec{n_1} \neq \vec{n_2}$ with $|\vec{n_1}| = |\vec{n_2}| = n$). Same comment for $Q_k(x)$, $\Phi_k(z)$ and $\Psi_k(z)$.} {\color{brown} Agreed, Maksim $\to$ Kenta: Try to clarify the use of the natural path in Section 2, so done if you are OK.}} 
defined in Definition~\ref{defn of type II} can be expressed as
\begin{align}
    P_{\vec{n}}(x) &= p_n(x) - \bm{v}_2^\top \bm{M}_{22}^{-1} \bm{p}_2(x)\\
    &= \frac{1}{\det \bm{M}_{22}} \det
\begin{pmatrix}
\bm{M}_{22}  & \bm{p}_2(x)  \\
\bm{v}_2^\top & p_n(x)
\end{pmatrix},
\end{align}
where $\bm{v}_2$ is defined as in  \eqref{defn of v} and \eqref{defn of pi phi v, 1-2}, $\bm{M}_{22}$ is defined in \eqref{defn of M} and \eqref{defn of M, H 4 blocks}, and $\bm{p}_2(x)$ is defined as in \eqref{defn of pi} and \eqref{defn of pi phi v, 1-2}.
\end{theorem}

\begin{theorem}\label{type I reduction}
The type I linear form $Q_{\vec{n}}(x)$ defined in Definition~\ref{defn of type I} can be expressed as
\begin{align}\label{eq:Q_sol_general}
Q_{\vec{n}}(x) &= \bm{e}_n^\top \bm{M}^{\top-1} \bm{q}(x),\\
&= -\frac{1}{\det \bm{M}} \det
\begin{pmatrix}
 \bm{M} & \bm{e}_n\\
\;\bm{q}(x)^\top  &  0
\end{pmatrix},
\end{align}
where $\bm{e}_n$ is the standard basis column vector $(0, \cdots, 0, 1)^\top$, $\bm{M}$ is defined as in \eqref{defn of M}, and $\bm{q}(x)$ is defined as in \eqref{defn of phi}.
\end{theorem}

\begin{theorem}\label{kernel reduction}
Let $K(x,y)$ and $K_0(x,y)$ be the Christoffel-Darboux kernels defined as in \eqref{defn of CD kernel} and \eqref{defn of CD kernel for p} respectively.
Then, one has
\begin{equation} \label{reduction formula equation}
    K(x,y)-K_0(x,y) = \bm{r}(y)^\top \bm{M}_{22}^{-1} \bm{p}_2(x),
\end{equation}
where the vector $\bm{r}(y)$ is the projection of the companion weighted basis functions off the reference space,
\begin{equation} \label{defn of q}
\bm{r}(y) = \bm{q}_2(y) - \int_{\R} K_0(x,y) \bm{q}_2(x) \, \dd x,
\end{equation}
$\bm{M}_{22}$ is defined in \eqref{defn of M} and \eqref{defn of M, H 4 blocks}, and $\bm{p}_2(x)$ is defined as in \eqref{defn of pi} and \eqref{defn of pi phi v, 1-2}.
\end{theorem}

As an application of Theorem~\ref{kernel reduction}, we compute the multiple Hermite kernel for a general multi-index. Let $\vec n =(n_1,\ldots,n_r)\in\mathbb N^r$, $n=|\vec n|$, and consider the weights
\begin{equation}\label{mult hermit weights}
    w^{(j)}(x)=e^{a_jx-x^2},\qquad a_1=0,\qquad a_i\neq a_j \quad \text{for} \quad i\neq j,\qquad j=1,\ldots,r .
\end{equation}
For the reduction, the Gaussian weight $w^{(1)}(x)=e^{-x^2}$ is taken as the reference weight, and the remaining weights $w^{(2)},\ldots,w^{(r)}$ form the companion family. There are many studies of this classical family of multiple Hermite polynomials; see \cite{ApBrVA,BleKuij-HerLag} and the references therein. The explicit formulae for type II multiple Hermite polynomials are well known; see, for example, \cite[Section~23.5]{Ismail}.
\begin{align} \label{HermiteTypeII} H_{\vec n}(x;a_1,\ldots,a_r) &= \left(-\frac{1}{2}\right)^{|\vec n|} \sum_{\ell_1=0}^{n_1}\cdots\sum_{\ell_r=0}^{n_r} 2^{\ell_1+\cdots+\ell_r} \prod_{q=1}^r \frac{(-n_q)_{\ell_q}}{\ell_q!}\, a_q^{\,n_q-\ell_q}\, \widehat{H}_{\ell_1+\cdots+\ell_r}(x), \end{align} 
where $\widehat{H}_m(x)$ denotes the monic Hermite polynomial \begin{align} \label{HermiteMonic} \widehat{H}_m(x)= \sum_{k=0}^{\lfloor m/2\rfloor} (-1)^k \frac{(-m)_{2k}}{k!}\, \frac{x^{m-2k}}{2^{2k}}, \qquad m\in\mathbb N_0 .
\end{align}
Explicit formulae for type I multiple Hermite polynomials were obtained more recently in \cite{BranquinhoDiazFoulquieMorenoManas2025}.
 
\begin{multline}
\label{HermiteTypeI}
H^{(i)}_{\vec n}(x;a_1,\ldots,a_r)
=
\frac{(-1)^{n_i-1}}{\sqrt{\pi}(n_i-1)!}\,
\frac{2^{|\vec n|-1}e^{-a_i^2/4}}
{\prod_{\substack{q=1\\ q\neq i}}^r (a_i-a_q)^{n_q}}
\\
\times
\sum_{\substack{\ell_1,\ldots,\ell_r\ge 0\\
\ell_1+\cdots+\ell_r\le n_i-1}}
\frac{(-n_i+1)_{\ell_1+\cdots+\ell_r}}{\ell_i!}
\prod_{\substack{q=1\\ q\neq i}}^r
\frac{(n_q)_{\ell_q}}
{\ell_q!\,(a_q-a_i)^{\ell_q}}\,
\widehat{H}_{n_i-1-\ell_1-\cdots-\ell_r}\!\left(\frac{a_i}{2}\right)
x^{\ell_i},
\end{multline}
 for $i\in\{1,\ldots,r\}$.

The preceding formulae could in principle be obtained from Theorems~\ref{typeIIreduction} and~\ref{type I reduction}. However, the corresponding multiple Hermite Christoffel--Darboux kernel, to the best of our knowledge, has not previously appeared in the literature. Therefore, we only present a closed-form expression for the latter.

\begin{corollary}
\label{cor-general-Hermite-kernel}
Consider a general multi-index
\(\vec n=(n_1,\ldots,n_r)\in\mathbb N^r\) and the weights given in \eqref{mult hermit weights}.For $j=2,\ldots,r$, define
\begin{equation}\label{eq:Lj-general-Hermite}
    t_j=\frac{a_j}{2}, \quad \mbox{and} \quad L_j(z)=\left(\frac{z}{t_j}\right)^{n_1}\prod_{\substack{i=2\\ i\neq j}}^r\left(\frac{z-t_i}{t_j-t_i}\right)^{n_i}.
\end{equation}
For $0\leq q\leq n_j-1$, define the coefficients $v_s^{(j,q)}$ by
\begin{equation}\label{eq:Pi-jq}
    \sum_{s=0}^{N-1}v_s^{(j,q)}z^{n_1+s} := L_j(z) \frac{(z-t_j)^q}{q!}\sum_{\mu=0}^{n_j-1-q}\frac{1}{\mu!}\left[\frac{\dd^\mu}{\dd z^\mu}\frac{1}{L_j(z)}\right]_{z=t_j}(z-t_j)^\mu,
\end{equation}
where $N = n-n_1$.
For $0 \leq b \leq n_j-1$, $0 \leq m \leq n_1-1$, and $0 \leq s \leq N-1$, introduce
\begin{equation}\label{eq:U-Gamma-explicit}
    U^{(j)}_{b,s}=\frac{e^{-t_j^2}}{\sqrt{\pi}}\sum_{q=b}^{n_j-1}(-1)^{q-b}2^q t_j^{q-b}\binom{q}{b}v_s^{(j,q)},\quad
    \Gamma_{m,s}=\sum_{j=2}^r\sum_{q=0}^{\min(m,n_j-1)}q!\binom{m}{q}t_j^{m-q}v_s^{(j,q)}.
\end{equation}
Then, the multiple Hermite kernel is given by
\begin{multline}\label{eq:kernel-simplified-explicit}
    K_{\vec n}(x,y)=e^{-y^2}\sum_{m=0}^{n_1-1}\frac{2^m\widehat H_m(y)}{\sqrt{\pi}\,m!}\left(\widehat H_m(x)-\sum_{s=0}^{N-1}\Gamma_{m,s}\widehat H_{n_1+s}(x)\right)\\
    +\sum_{s=0}^{N-1}\sum_{j=2}^r\sum_{b=0}^{n_j-1}U^{(j)}_{b,s}\widehat H_{n_1+s}(x)\widehat H_b(y)e^{a_jy-y^2},
\end{multline}
where $\widehat H_m(x)$ is the monic classical Hermite polynomial.
\end{corollary}

%%%%%%%%%%%%%%%%

\subsection{MOPUC}

So far, we have focused on MOPRL. However, analogous results also hold for MOPUC, since the proofs of the theorems for MOPRL rely only on the linear algebraic structure of the Hilbert space of polynomials $\mathcal{P}^{\mathbb{R}}_{n}$; see Section~\ref{section proofs}. Therefore, the proofs of the corresponding theorems for MOPUC, presented in Section~\ref{appendix mopuc}, are similar to those in Section~\ref{section proofs}.

\begin{theorem}\label{thm:typeII}
A type~II multiple orthogonal polynomial $\Phi_{\vec{n}}(z)$ on the unit circle $\TT$ defined in Definition~\ref{def:typeII MOPUC} admits the following reduction formula,
\begin{align}\label{eq:typeII-expansion}
\Phi_{\vec{n}}(z) &= \varphi_{n}(z) - \bm{u}_2^\top  \bm{\mathcal M}_{22}^{-1} \bm{\varphi}_2(z)\\
&= \frac{1}{\det \bm{\mathcal M}_{22}} 
\det\begin{pmatrix}
\bm{\mathcal M}_{22} & \bm{\varphi}_2(z)\\
\bm{u}_2^\top & \varphi_{n}(z)
\end{pmatrix},
\end{align}
where $\varphi_n(z)$ is a standard OPUC, $\bm{u}_2$ is a companion part of an $n$-column vector of the mixed-moments of $\varphi_n(z)$ as in \eqref{eq:v-def} and \eqref{eq:block-structure}, $\bm{\mathcal{M}}_{22}$ is a companion $(n-n_1) \times (n-n_1)$ matrix of mixed-moments defined in \eqref{eq:M-def} and \eqref{eq:block-structure}, and $\bm{\varphi}_2(z)$ is a companion part of an $n$-column vector of standard OPUC defined in \eqref{eq:pi-primary} and \eqref{eq:block-structure}.
\end{theorem}

\begin{theorem}\label{thm:typeI}
The type~I linear form $\Psi_{\vec{n}}(z)$ defined in Definition~\ref{def:typeI MOPUC} admits the following reduction formula,
\begin{align}\label{eq:typeI-explicit}
\Psi_{\vec{n}}(z) &= \bm{e}_n^\top \overline{\bm{\mathcal{M}}}^{\top -1} \bm{\psi}(z)\\
&= -\frac{1}{\det \overline{\bm{\mathcal{M}}}}
\det\begin{pmatrix}
\overline{\bm{\mathcal{M}}} & \bm{e}_n\\
\bm{\psi}(z)^\top & 0
\end{pmatrix},
\end{align}
where $\bm{e}_n$ is the standard basis column vector $(0, \cdots, 0, 1)^\top$, $\bm{\mathcal{M}}$ is a $n \times n$ matrix of mixed-moments defined in \eqref{eq:M-def}, and $\bm{\psi}(z)$ is an $n$-column vector of weighted OPUC defined in \eqref{eq:chi-stacked}.
\end{theorem}

\begin{theorem}\label{thm:kernel}
Let $\mathcal K$ and $\mathcal K_0$ be the Christoffel-Darboux kernels for MOPUC and OPUC, defined as in \eqref{eq:Kn-def} and \eqref{eq:Szego-kernel} respectively. Then, one has
\begin{equation}\label{eq:kernel-reduction}
\mathcal K(z,\xi) - \mathcal K_0(z,\xi) = \overline{\bm{\pi}(\xi)}^\top \bm{\mathcal M}_{22}^{-1} \,
\bm{\varphi}_2(z),
\end{equation}
where the projection $\bm{\pi}(\xi)$ is defined by
\begin{equation}\label{eq:q-def}
\bm{\pi}(\xi) = \bm{\psi}_2(\xi) - \int_{\TT}  \overline{\mathcal K_0 (s,\xi)} \, \bm{\psi}_2(s) \, \frac{|\dd s|}{2\pi},
\end{equation}
with $\bm{\psi}_2(\xi)$ defined in \eqref{eq:chi-stacked} and \eqref{eq:block-structure}.
\end{theorem}

To illustrate the explicit formulae obtained above and to further investigate
zeros of type II MOPUC, we consider weights given by characteristic functions
of arcs of the unit circle.  This is a natural class of examples in the
Toeplitz/OPUC setting, whose connection with random-matrix gap probabilities
is well known.  For instance, if \(U\) is drawn from CUE, then the probability
that all eigenangles of \(U\) lie in a measurable set \(B\subset[0,2\pi)\) is
the Toeplitz determinant generated by \(\chi_B\); see
\cite{Mehta04,CharlierClaeys15}.  Similarly, the GUE bulk gap probability on
\((0,2s)\) is obtained as the large-\(n\) limit of the \(n\times n\) Toeplitz
determinant generated by the characteristic function of the arc
\(2s/n\leq \theta\leq 2\pi-2s/n\); see \cite{Krasovsky04}
and also \cite{Marchal20}.

In the setting of multiple orthogonality on the unit circle, systems of
measures supported on arcs were considered by M\'inguez Ceniceros and
Van Assche in \cite{MCVA08}.  Recall that, for ordinary OPUC associated
with a positive measure on \(\mathbb T\) with infinite support, all zeros of
the corresponding monic orthogonal polynomials lie in the open unit disk.
The examples in \cite{MCVA08} show that the zero behavior of type II MOPUC
can differ substantially from this classical picture.  For two disjoint arcs
\cite[Example~1]{MCVA08}, the zeros appear to lie on curves connecting the
endpoints of the supporting arcs, rather than accumulating close to the arcs
themselves.  For overlapping arcs \cite[Example~2]{MCVA08}, the zeros may
even leave the unit disk.  In the case where one measure is supported on the
full circle and the other on a single arc, the type II polynomial has a zero
of large multiplicity at the origin, while the remaining zeros are governed
by the geometry of the arc.  In the diagonal case, this phenomenon is made
explicit in \cite[Proposition~3.1]{MCVA08}, where the nontrivial factor is
expressed in terms of a Legendre polynomial.

The examples below are in the same spirit, but are chosen so that our general formulae lead to completely explicit expressions.  In
particular, they make visible a simple resonance mechanism by which zeros
of type II MOPUC can escape to infinity.

Suppose the multi-index $\vec{n}$ is given by
\begin{equation} \label{multi-index MOPUC eg 1}
\vec{n} = (n-r, 1, \cdots, 1) \in \N^{r+1},
\end{equation}
and the weights of orthogonality on $\mathbb T \ni z = e^{i \theta}$ are given by
\begin{equation} \label{weights MOPUC eg 1}
\omega^{(1)}(e^{i\theta}) := 1,\quad
\omega^{(j+1)}(e^{i\theta}) := \mathbf{1}_{E_j}(\theta),\quad
E_j=[\alpha_j,\beta_j] \subseteq [0, 2\pi), \quad j=1,\cdots,r,
\end{equation}
where we assumed the interval of angles $E_1, \cdots, E_r$ are pairwise disjoint with equal length $\Delta$, i.e., 
\[
\Delta := |E_j| = \beta_j - \alpha_j, \quad j = 1, \cdots, r.
\]
We also denote the center of each angle by 
\[
\theta_j := \frac{\alpha_j + \beta_j}{2},
\]
and denote the center of the corresponding subarc of $\mathbb T$ by 
\begin{align} \label{defn of x_j}
x_j = e^{i \theta_j}.
\end{align}
Then, one can obtain the following corollaries to express the corresponding MOPUC, $\Phi_{\vec{n}}(z)$ and $\Psi_{\vec{n}}(z)$, and their kernel, using 
\begin{align}\label{defn of a_l}
a_\ell = \begin{dcases}
\frac{1}{\ell \pi} \sin \left( \frac{\ell \Delta}{2} \right), & \ell \neq 0,\\
\frac{\Delta}{2 \pi}, & \ell = 0,
\end{dcases}
\end{align}
and the degree $k$ elementary symmetric polynomial,
\begin{equation}\label{defn of esp}
e_k(x_1, \dots, x_n) = \sum_{1 \leq j_1 < j_2 < \dots < j_k \leq n} x_{j_1} x_{j_2} \dots x_{j_k}, \quad 1 \leq k \leq n.
\end{equation}
In the following, we also use the notation 
\begin{equation}\label{eq: lq}
l_q := n-r+q-1,    \qquad 1\leq q \leq r. 
\end{equation} 
\begin{corollary}\label{cor: MOPUC type II eg 1} Assume $\Delta \notin \pi \mathbb Q$.
The type~II MOPUC defined by \eqref{multi-index MOPUC eg 1} and \eqref{weights MOPUC eg 1} under the equal-length assumption is expressed by
\begin{align} \label{MOPUC type II eg 1}
\Phi_{\vec{n}}(z) &= z^n - \sum_{q=1}^r \left( \frac{a_n}{a_{l_q}} (-1)^{r-q} e_{r-q+1}(x_1, x_2, \cdots, x_r) \right) z^{l_q}
\end{align}
where $a_\ell$ and $l_q$ are as in \eqref{defn of a_l} and \eqref{eq: lq}, and $x_j$ is defined in \eqref{defn of x_j}.
Moreover, $\Phi_{\vec{n}}(z)$ has a zero of multiplicity $n-r$ at the origin and $r$ zeros on the complex plane, say $\zeta_1, \zeta_2, \cdots, \zeta_r$. We have
\begin{align} \label{eg 1 escaping}
\max_{1\leq j \leq r} |\zeta_j| \geq \left( \left| \frac{n-r+q-1}{n} \frac{\sin \left(\frac{n\Delta}{2} \right)}{\sin \left(\frac{(n-r+q-1) \Delta}{2} \right)}\right| \frac{|e_{r-q+1}(x_1, \cdots, x_r)|}{\binom{r}{q-1}}\right)^{\frac{1}{r-q+1}}
\end{align}
for $q = 1, \cdots, r$. In particular, $q = 1$ in \eqref{eg 1 escaping} gives
\begin{equation}\label{maxzej}
        \max_{1\leq j \leq r} |\zeta_j| \geq \left| \frac{n-r}{n} \frac{\sin \left(\frac{n\Delta}{2} \right)}{\sin \left(\frac{(n-r) \Delta}{2} \right)}\right|^{\frac1r}.
\end{equation}

\end{corollary}

\begin{remark}\label{remark MOPUC eg 1} Here we make some observations about the zeros of arc-indicator type II MOPUC under the equal-length assumption.  \phantom{newline}
\begin{enumerate}
\item The estimate \eqref{eg 1 escaping} gives a simple mechanism for zero escape in the equal-length regime. Suppose that $r$ and $\Delta \notin \pi \mathbb{Q}$ are fixed. Since the sequence
$
\{m\frac{\Delta}{2}\pmod {2\pi}\}_{ m\in\mathbb N}
$
is dense in $[0,2\pi)$, we may choose $m_k$ such that
$\frac{m_k\Delta}{2}\to 0 \pmod {2\pi}$. In particular, $
\sin\left(\frac{m_k\Delta}{2}\right)\to 0.
$
Setting $n_k=m_k+r$, we obtain
$
\sin\left(\frac{(n_k-r)\Delta}{2}\right)\to 0.
$
Moreover,
$
\sin\left(\frac{n_k\Delta}{2}\right)
\to
\sin\left(\frac{r\Delta}{2}\right),
$
which is nonzero because $\Delta/\pi\notin\mathbb Q$. If necessary, choose a tail of this subsequence and relabel so that for all $k$ it holds that $ \left|\sin\left(\frac{n_k\Delta}{2}\right)\right|\geq c>0$ for some positive constant $c$. Then the lower bound obtained from the case $q=1$ in \eqref{eg 1 escaping} tends to infinity, and thus at least one zero of $\Phi_{\vec n_k}$ escapes to infinity. 
\item At the same time, for a fixed finite degree $n$, one can also produce the same effect by tuning the geometric parameters of the arcs. Namely, one may choose the common arc length $\Delta$ so that $\sin\left(\frac{(n-r)\Delta}{2}\right)$ is very small while $\sin\left(\frac{n\Delta}{2}\right)$ remains nonzero. In this finite-$n$ interpretation, the natural parameters are the arc centers $x_1,\ldots,x_r$, the number of arcs $r$, and the common arc length $\Delta$;  see Figures~\ref{fig:mopuc-arc-supports-n8-I} and~\ref{fig:mopuc-zero-escape-n8-I}.
\item In \cite[Example~2]{MCVA08}, it is observed that some zeros of the
type II MOPUC of fixed degree lie outside the unit disk.  However, their
overlapping-arc example does not address the possibility that, under a
limiting choice of geometric parameters, some zeros may tend to infinity.
This phenomenon appears to be observed here for the first time in the MOPUC
setting.  
\end{enumerate}
\end{remark}

\begin{corollary} \label{cor: MOPUC type I eg 1}
The type I linear form $\Psi_{\vec{n}}(z)$ defined by \eqref{multi-index MOPUC eg 1} and \eqref{weights MOPUC eg 1} under the equal-length assumption is expressed by
\begin{align} \label{MOPUC type I eg 1}
\Psi_{\vec{n}}(z) &= \sum_{j=1}^r c_j \mathbf{1}_{E_j}(\arg z) + \sum_{l=0}^{n-r-1} \left( \sum_{k=1}^r \frac{c_k \, a_{l}}{x_k^l} \right) z^l
\end{align}
where
\begin{align} \label{eg 1 defn of cj}
c_j = \frac{(-1)^{r-1} x_j^{n-2}}{a_{n-1}} \frac{\prod_{m=1}^r x_m}{\prod_{\substack{m=1 \\ m \neq j}}^r (x_j - x_m)}
\end{align}
with $a_\ell$ and $x_j$ defined in \eqref{defn of a_l} and \eqref{defn of x_j} respectively.
\end{corollary}

\begin{corollary} \label{cor: MOPUC kernel eg 1}
Let $\mathcal K$ be the Christoffel-Darboux kernel for MOPUC defined by \eqref{multi-index MOPUC eg 1} and \eqref{weights MOPUC eg 1} under the equal-length assumption, and let $\mathcal K_0$ be the Christoffel-Darboux kernel for OPUC defined by $\omega^{(1)}(z) = 1$ on $\TT$. Then, one has
\begin{equation} \label{MOPUC kernel eg 1}
    \mathcal K(z,\xi) - \mathcal K_0(z,\xi) = \sum_{p=1}^r \left( \mathbf{1}_{E_p} (\arg \xi) - \sum_{j=0}^{n-1} \xi^{-j} a_j x_p^j \right) \left( \frac{x_p^{-(n-r)}}{\prod_{\substack{m=1 \\ m \neq p}}^r (x_p - x_m)} \sum_{k=1}^r (-1)^{r-k} e_{r-k}(\hat{x}_p) \frac{z^{l_k}}{a_{l_k}}\right) 
\end{equation}
where $x_j$, $a_\ell$, and $l_q$, are defined in \eqref{defn of x_j}, \eqref{defn of a_l} and \eqref{eq: lq}. Here $e_s (\hat{x}_p)$ denotes the elementary symmetric polynomial of degree $s$ in the $r-1$ variables $x_1, x_2, \cdots, x_{p-1}$, $x_{p+1}, \cdots, x_r$.
\end{corollary}

Moreover, if we further assume that the weights $\omega^{(2)}(z), \cdots, \omega^{(r+1)}(z)$ are assigned on the equal-length and equally-spaced subarcs of $\mathbb T$, i.e., we assume that each center of subarcs of $\TT$ corresponding to $E_1, \cdots, E_r$ are given by
\begin{align} \label{equal space MOPUC eg 2}
x_j = x_1 \rho^{j-1}, \quad \rho = e^{\frac{2\pi i}{r}}, \quad j=1, \cdots, r
\end{align}
we can simplify the above corollaries.

\begin{corollary} \label{cor: MOPUC type II eg 2}
Under the equally-spaced assumption \eqref{equal space MOPUC eg 2}, suppose that $a_n a_{n-r}\neq 0$. Then \eqref{MOPUC type II eg 1} simplifies to
\begin{equation} \label{MOPUC type II eg 2}
\Phi_{\vec n}(z)=z^n-\left(\frac{a_n}{a_{n-r}}x_1^r\right)z^{n-r}=z^{n-r}\left(z^r-\frac{a_n}{a_{n-r}}x_1^r\right).
\end{equation}
Hence $\Phi_{\vec n}$ has a zero of multiplicity $n-r$ at the origin and $r$ simple nonzero zeros
\begin{equation}\label{zek}
    \zeta_k=\left|\frac{n-r}{n}\frac{\sin\left(\frac{n\Delta}{2}\right)}{\sin\left(\frac{(n-r)\Delta}{2}\right)}\right|^{1/r}\exp\left(i \vartheta_k  \right), \qquad \vartheta_k=\theta_1+\frac{\arg(a_n/a_{n-r})+2\pi k}{r}, \quad k=0,\ldots,r-1.
\end{equation}
Here $\arg(a_n/a_{n-r})\in\{0,\pi\}$, since $a_n/a_{n-r}$ is real.

\end{corollary}

\begin{remark} \label{remark MOPUC eg 2} Here we make some observations about the zeros of arc-indicator type II MOPUC under the equal-length and equally-spaced assumptions. \phantom{newline}
\begin{enumerate}
\item 
In the equally spaced regime \eqref{equal space MOPUC eg 2}, the escaping mechanism becomes completely explicit. By \eqref{MOPUC type II eg 2},   if $r$ and $\Delta$ are fixed, and if $\{n_k\}_{k\in\mathbb N}$ is a sequence of degrees such that $\sin\left(\frac{(n_k-r)\Delta}{2}\right)\to 0$, while $\left|\sin\left(\frac{n_k\Delta}{2}\right)\right|\geq c>0$, then all $r$ nonzero zeros of $\Phi_{\vec n_k}$ escape to infinity simultaneously.
\item
The same resonance can also be observed at fixed finite $n$. The geometric parameters are the common arc length $\Delta$, the number of arcs $r$. By choosing $\Delta$ so that $\sin\left(\frac{(n-r)\Delta}{2}\right)$ is very small while $\sin\left(\frac{n\Delta}{2}\right)$ remains nonzero, one obtains $r$ large nonzero zeros already at finite degree; see Figures~\ref{fig:mopuc-arc-supports-n8-II} and~\ref{fig:mopuc-zero-escape-n8-II}.
\end{enumerate}
\end{remark}

Using the explicit formula \eqref{MOPUC type II eg 2}, one can also study the set of nontrivial zeros of $\Phi_{\vec n}(z)$ and their asymptotic density as described in the next Proposition.

\begin{proposition}\label{thm:fixed-r-ray-density}
In the equally spaced MOPUC arc example, fix $r\ge 1$ and assume
$0<\Delta<2\pi/r$ and $\Delta/\pi\notin \mathbb Q$.
Write the set of $r$ nontrivial zeros of the degree $n$ polynomial \eqref{MOPUC type II eg 2} as $\mathcal Z^{(n)}_{r}$.
Then, these $r$ zeros of $\Phi_{\vec n}(z)$ are asypmtotically dense in the $2r$ rays on the complex plane, i.e.,
\[
\bigcap_{m=r+1}^{\infty} \overline{\bigcup_{n \geq m} \mathcal Z^{(n)}_{r}} = \bigcup_{\ell=0}^{2r-1}
  \{\varrho \, e^{i(\theta_1+\pi\ell/r)}:\varrho\ge 0\}.
\]
\begin{comment}
For $n>r$, let $\mathcal Z_r$ be the set of nonzero zeros of the corresponding type II 
polynomials $\Phi_{\vec{n}}$, taken over all $n=r+1,r+2,\ldots$. If
$x_1=e^{i\theta_1}$, then
\[
  \overline{\mathcal Z_r}
  =
  \bigcup_{\ell=0}^{2r-1}
  \{\varrho e^{i(\theta_1+\pi\ell/r)}:\varrho\ge 0\}.
\]
\end{comment}
\end{proposition}

\begin{proof}
Put $k=n-r$. By \eqref{zek}, the nontrivial zeros are the roots of
$z^r=A_{k+r,r}(\Delta)x_1^r$, where
\[
  A_{k+r,r}(\Delta)
  =\frac{k}{k+r}\frac{\sin((k+r)\Delta/2)}{\sin(k\Delta/2)}
  =\frac{k}{k+r}\left(\cos\frac{r\Delta}{2}
    +\sin\frac{r\Delta}{2}\cot\frac{k\Delta}{2}\right).
\]
Since $\Delta/\pi$ is irrational, $k\Delta/2 \pmod {\pi}$ is dense in $(0,\pi)$ . Also
$\sin(r\Delta/2)\ne0$. Thus, the last parenthesis takes every real value
along subsequences of $k$, and since $k/(k+r)\to1$, every $y\in\R$ is a
subsequential limit of $A_{k+r,r}(\Delta)$.

Let $\varrho \ge 0$ and $\ell \in \{0, \ldots, 2r-1\}$, choose a subsequence with
$A_{k_\ell+r,r}(\Delta)\to(-1)^\ell\varrho^r$. 
Then, $z^r \to (-1)^\ell \varrho^r x_1^r$ and the $\ell$-th root satisfies $z_\ell \to \varrho \, e^{i(\theta_1 + \pi \ell/r)}$.
Conversely, $A_{k+r,r}(\Delta)$ is real for all $k$, so every nontrivial zero lies on one of these $2r$ rays.
\end{proof}

Moreover, one can also show that $ \bigcup^{\infty}_{r=1}\bigcap_{m=r+1}^{\infty} \overline{\bigcup_{n \geq m} \mathcal Z^{(n)}_{r}}$ is dense in  $\C$ after replacing $\Delta \mapsto \Delta_r := \eta/r$ where $0<\eta<2\pi$ and $\eta/\pi \notin \mathbb Q$. This should be compared with Stahl's  OPRL example, see \cite[Theorems 1 and 2]{Stahl84} and \cite[Theorem 1]{Stahl91},  where the asymptotic density is obtained for a single, but more intricate, orthogonality measure. Namely, Stahl considered orthogonal polynomials \(q_n\) associated with the measure
\[
\dd \mu(x) = \frac{(x - \cos \pi r_1)(x - \cos \pi r_2)}{\pi \sqrt{1 - x^2}} \dd x, \quad x \in [-1,1],
\]
where $r_1, r_2, r_1/r_2 \notin \mathbb Q$.    
It was observed that $n-2$ zeros lie in $[-1,1]$ and the remaining two zeros $Z(q_n)$ are asymptotically dense in $\C$, i.e. $\cap^{\infty}_{m=1} \overline{ \cup_{n \geq m} Z(q_n) } = \C$.

We now illustrate the finite-degree resonance discussed in Remarks \ref{remark MOPUC eg 1} and \ref{remark MOPUC eg 2} numerically. The table below gives
the parameters used in the corresponding figures for two regimes: Regime~I,
where the arcs have equal length but arbitrary centers, and Regime~II, the
symmetric specialization in which the arc centers are equally spaced. In both
cases, the common arc length is varied toward the first resonant value
\[
\Delta_{\mathrm{res}}:=\frac{2\pi}{n-r}.
\]
In the following examples we consider \(n=8\) and \(r=3\), so that $\Delta_{\mathrm{res}}=2\pi/5$.

\begin{table}[H]
\centering
\small
\begin{tabular}{c c c c c}
\toprule
Regime & Color & Arc centers $(\theta_1,\theta_2,\theta_3)$
& $\Delta/\Delta_{\mathrm{res}}$ & Observed zero with the largest modulus \\
\midrule

\multirow{5}{*}{I}
& Green & \multirow{5}{*}{$(17^\circ,\,97^\circ,\,201^\circ)$}
& $0.20$ & $0.9647$ \\
& Blue & & $0.40$ & $0.8415$ \\
& Gold & & $0.60$ & $0.4465$ \\
& Orange & & $0.80$ & $3.0348$ \\
& Red & & $0.99$ & $2.9585$ \\
\midrule

\multirow{5}{*}{II}
& Green &
\multirow{5}{*}{$(11^\circ,\,131^\circ,\,251^\circ)$}
& $0.20$ & $0.9647$ \\
& Blue & & $0.40$ & $0.8409$ \\
& Gold & & $0.60$ & $0.4351$ \\
& Orange & & $0.80$ & $0.9357$ \\
& Red & & $0.99$ & $2.6781$ \\

\bottomrule
\end{tabular}

\caption{Parameters used in
Figures~\ref{fig:mopuc-arc-supports-n8} and
\ref{fig:mopuc-zero-escape-n8} below, with
$n=8$, $r=3$.
The arc centers remain fixed within each regime, while the common arc
length is given as a fraction of $\Delta_{\mathrm{res}}$. The arc centers are rounded to the nearest degree.  For Regime~I, the final column gives the largest
modulus among the three computed nonzero zeros. For Regime~II, it
gives their common modulus. Expectedly, comparing the data in the last column of Regime I and Regime II also confirms \eqref{maxzej}.}
\label{tab:mopuc-n8-parameters}
\end{table}

%\newpage
\renewcommand\thesubfigure{\Alph{subfigure}}
\begin{figure}[ht!]
\centering
\begin{subfigure}{0.45\textwidth}
\centering
\includegraphics[width=\linewidth]{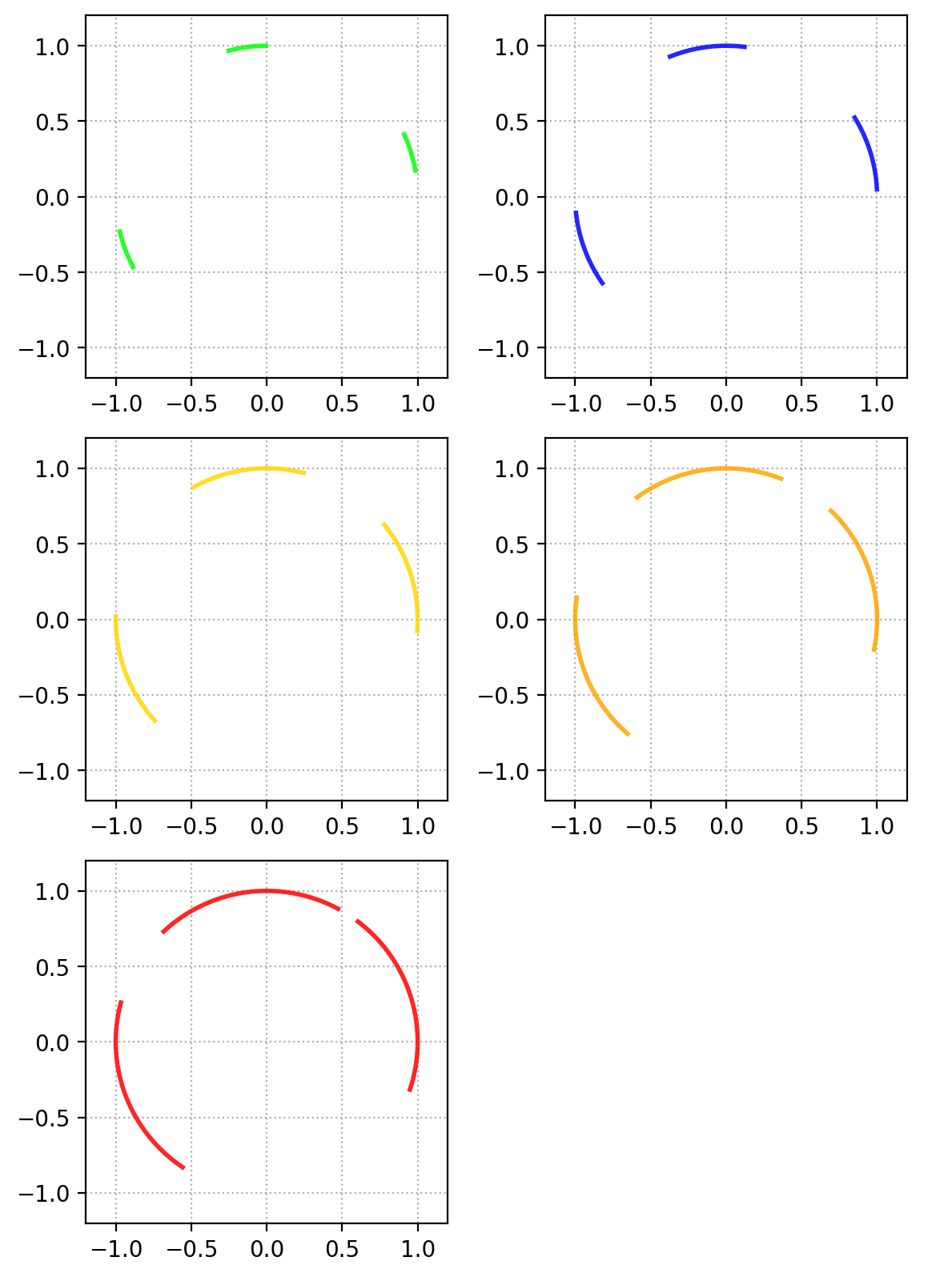}
\caption{ Regime I.}
\label{fig:mopuc-arc-supports-n8-I}
\end{subfigure}
\hfill
\begin{subfigure}{0.45\textwidth}
\centering
\includegraphics[width=\linewidth]{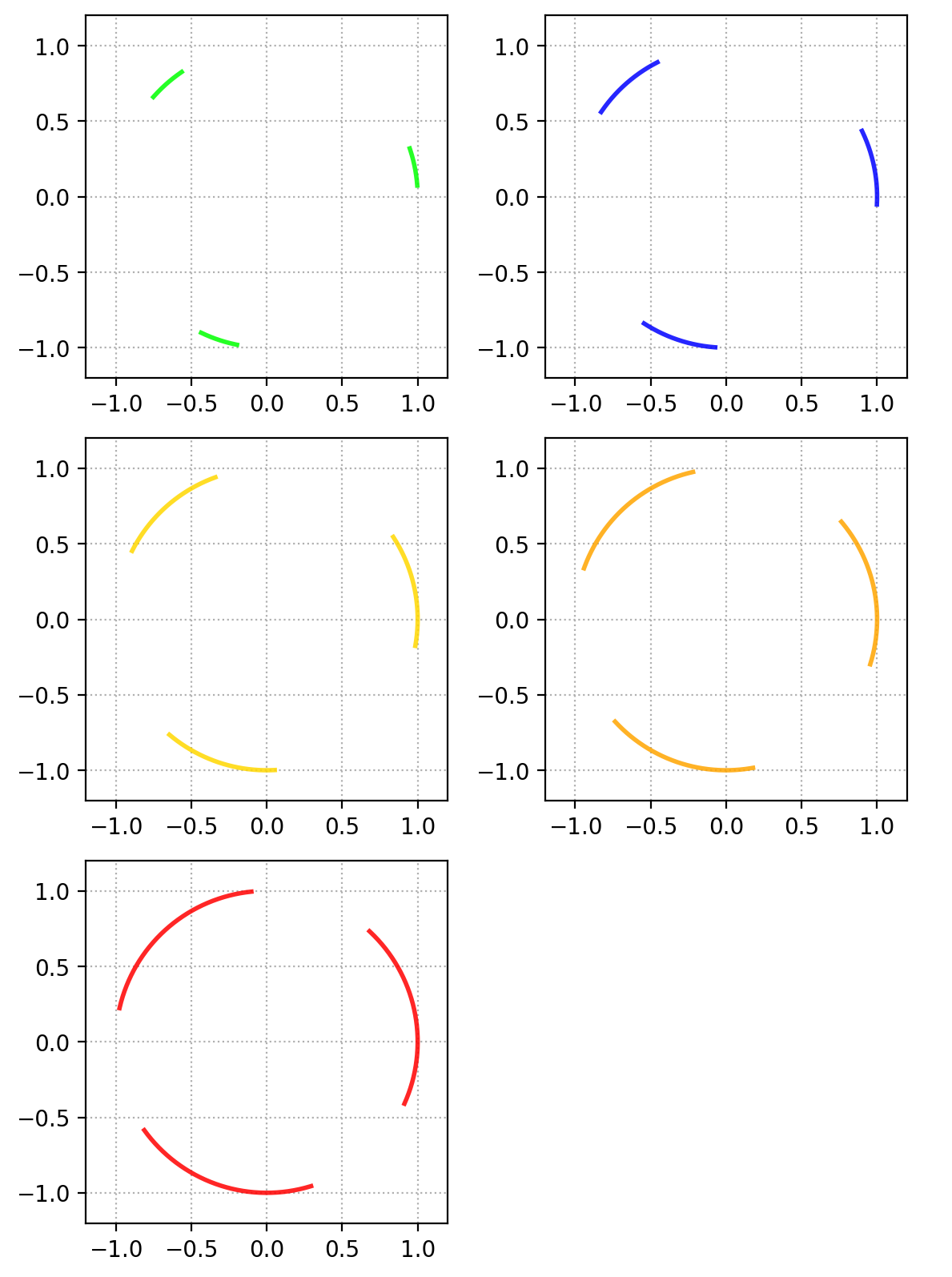}
\caption{ Regime II.}
\label{fig:mopuc-arc-supports-n8-II}
\end{subfigure}
\caption{Arc supports for the type II MOPUC examples. The geometric parameters are described in Table \ref{tab:mopuc-n8-parameters}.
}
\label{fig:mopuc-arc-supports-n8}
\end{figure}

\renewcommand\thesubfigure{\Alph{subfigure}}
\begin{figure}[ht!]
\centering
\begin{subfigure}{0.45\textwidth}
\centering
\includegraphics[width=\linewidth]{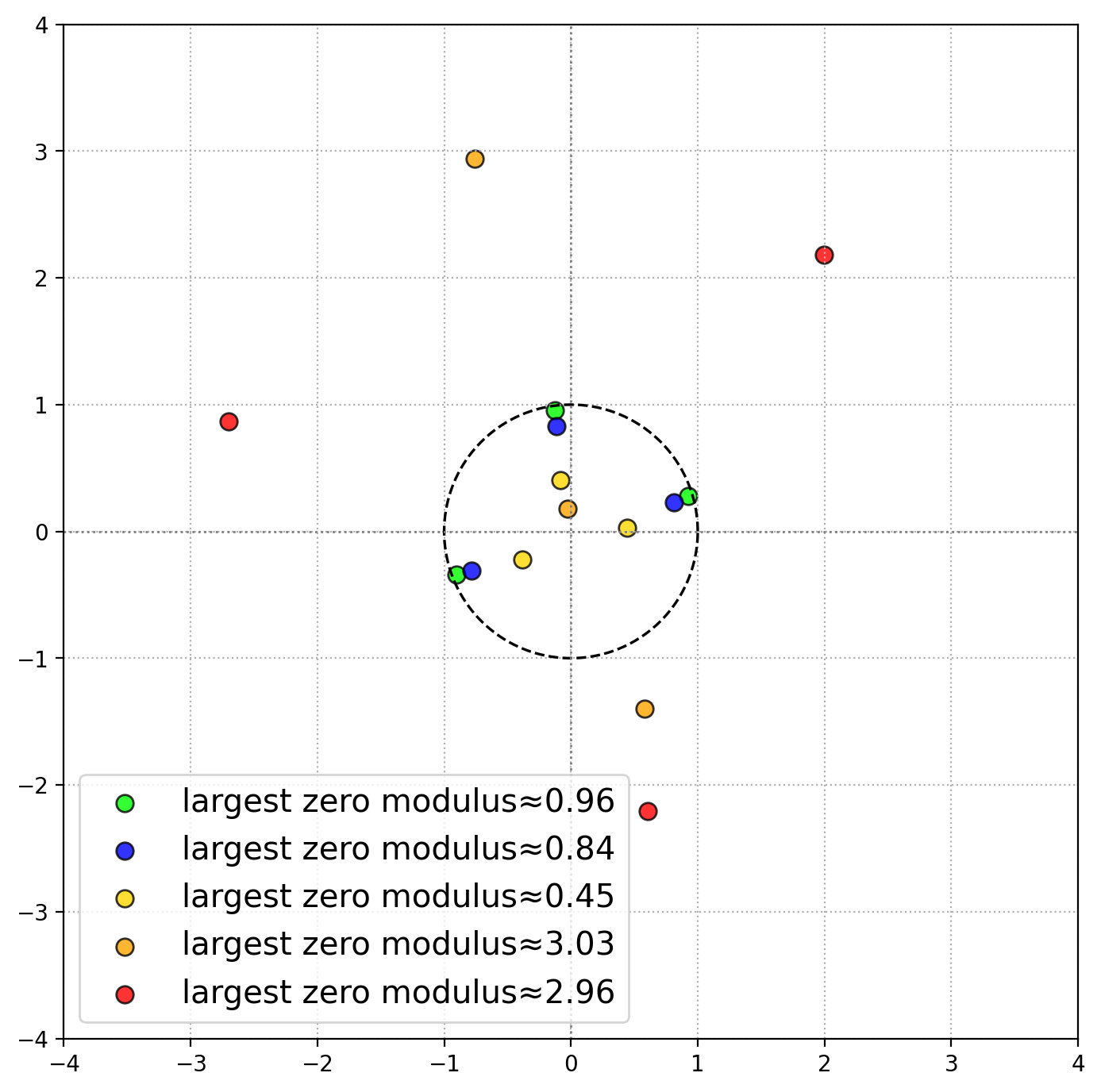}
\caption{Regime I. For each arc configuration, the legend gives the
largest modulus among the three nonzero zeros; see \eqref{maxzej}.}
\label{fig:mopuc-zero-escape-n8-I}
\end{subfigure}
\hfill
\begin{subfigure}{0.45\textwidth}
\centering
\includegraphics[width=\linewidth]{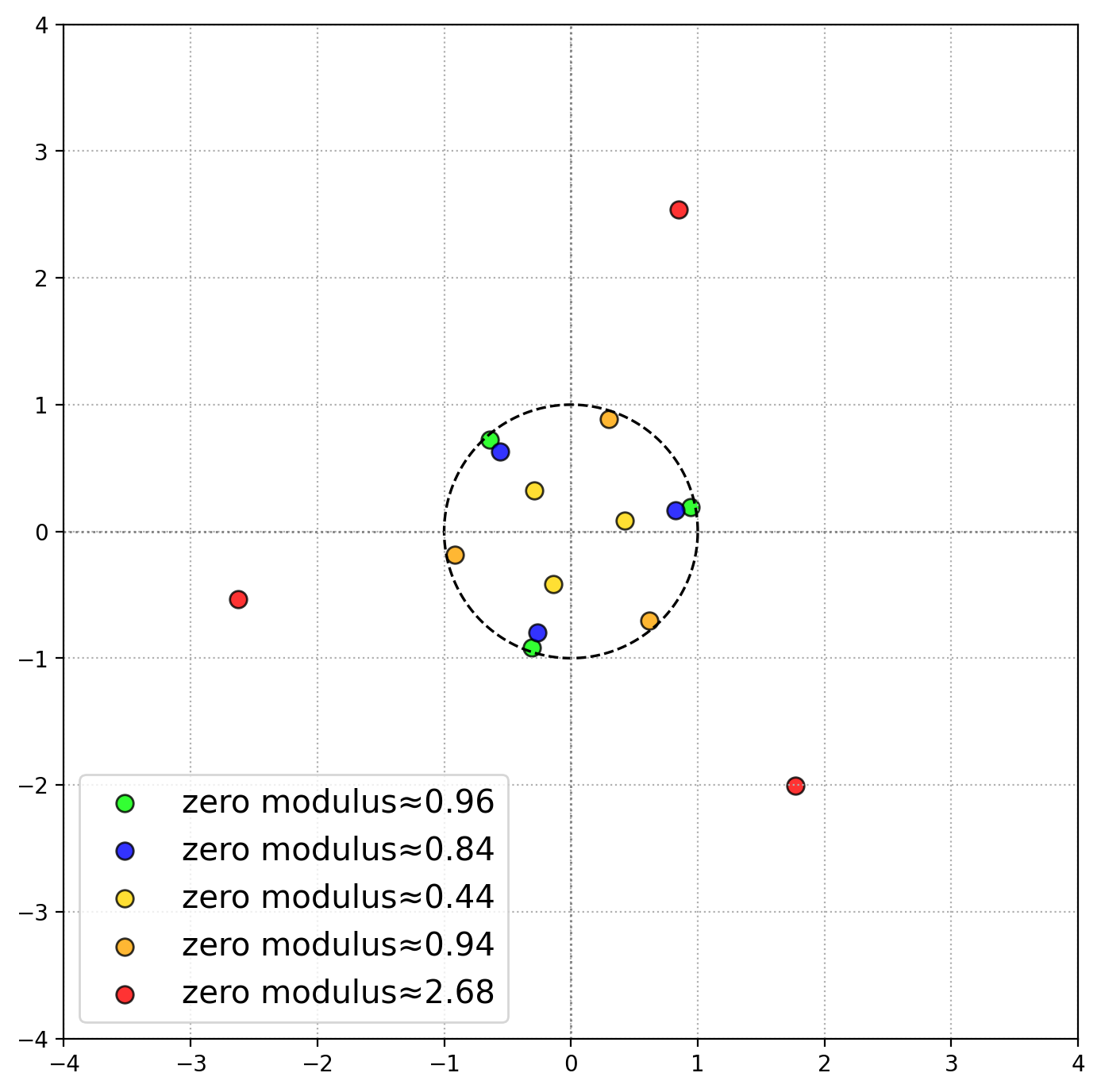}
\caption{Regime II. For each arc configuration, the legend gives the
common modulus of the three nonzero zeros; see \eqref{zek}.} 
\label{fig:mopuc-zero-escape-n8-II}
\end{subfigure}
\caption{Escape of the \textit{nonzero} type II MOPUC zeros.  The color of each zero set matches the color of the corresponding arc configuration shown in Figure~\ref{fig:mopuc-arc-supports-n8}. As expected, the zeros shown in Figure \ref{fig:mopuc-zero-escape-n8-II} lie on exactly six rays as prescribed in Proposition \ref{thm:fixed-r-ray-density}.}
\label{fig:mopuc-zero-escape-n8}
\end{figure}

\begin{corollary} \label{cor: MOPUC type I eg 2}
Under the equally-spaced assumption \eqref{equal space MOPUC eg 2}, \eqref{MOPUC type I eg 1} can be simplified as
\begin{align} \label{MOPUC type I eg 2}
\Psi_{\vec{n}}(z) = \frac{1}{r a_{n-1}} \left(\sum_{j=1}^r x_j^{n-1} \mathbf{1}_{E_j}(\arg z) + \sum_{l=0}^{n-r-1} \left( \sum_{k=1}^r a_l \, x_k^{n-l-1} \right) z^l \right).
\end{align}
\end{corollary}

\begin{corollary} \label{cor: MOPUC kernel eg 2}
Under the equally-spaced assumption \eqref{equal space MOPUC eg 2}, \eqref{MOPUC kernel eg 1} can be simplified as
\begin{align} \label{MOPUC kernel eg 2}
\mathcal K(z,\xi) - \mathcal K_0(z,\xi) = \frac{1}{r} \sum_{k=1}^r \left( \frac{z^{l_k}}{a_{l_k} x_1^{l_k}} \sum_{p=1}^r \left( \mathbf{1}_{E_p} (\arg \xi) - \sum_{j=0}^{n-1} \xi^{-j} a_j x_p^j \right) \rho^{-(p-1)l_k} \right).
\end{align}
\end{corollary}

\begin{comment}
\begin{proposition}[Ray density for fixed $r$]\label{thm:fixed-r-ray-density}
In the equally spaced MOPUC arc example, fix $r\ge 1$ and assume
$0<\Delta<2\pi/r$ and $\Delta/\pi\notin \mathbb Q$.  For $n>r$, let
$\mathcal Z_r^*$ be the set of nonzero zeros of the corresponding type II
polynomials $\Phi_{n,r}$, taken over all $n=r+1,r+2,\ldots$. If
$x_1=e^{i\theta_1}$, then
\[
  \overline{\mathcal Z_r^*}
  =
  \bigcup_{\ell=0}^{2r-1}
  \{\rho e^{i(\theta_1+\pi\ell/r)}:\rho\ge 0\}.
\]
\end{proposition}

\begin{proof}
Put $m=n-r$. By the symmetric formula, the nonzero zeros are the roots of
$z^r=A_{m+r,r}(\Delta)x_1^r$, where
\[
  A_{m+r,r}(\Delta)
  =\frac{m}{m+r}\frac{\sin((m+r)\Delta/2)}{\sin(m\Delta/2)}
  =\frac{m}{m+r}\left(\cos\frac{r\Delta}{2}
    +\sin\frac{r\Delta}{2}\cot\frac{m\Delta}{2}\right).
\]
Since $\Delta/\pi$ is irrational, $m\Delta/2 \pmod {\pi}$ is dense in $(0,\pi)$ . Also
$\sin(r\Delta/2)\ne0$. Thus, the last parenthesis takes every real value
along subsequences of $m$, and since $m/(m+r)\to1$, every $y\in\R$ is a
subsequential limit of $A_{m+r,r}(\Delta)$.
\end{proof}
\end{comment}

The proofs of these corollaries will be presented in Section~\ref{section example MOPUC}.

\section{Proof of Theorems (MOPRL)} \label{section proofs}

\subsection{Proof of Theorem~\ref{typeIIreduction}}

By Definition~\ref{defn of type II}, type II MOPRL $P_{\vec{n}}(x)$ is a monic polynomial of degree $n=|\vec{n}|$. 
As the space of polynomials of degree up to $n$ is a Hilbert space spanned by the monic orthogonal polynomials $\{ p_k (x) \}_{k=0}^n$ with respect to the reference weight $w^{(1)}(x)$, one can uniquely express $P_{\vec{n}}(x)$ by
\begin{equation}
    P_{\vec{n}}(x) = p_n(x) + \sum_{k=0}^{n-1} c_k p_k(x)
\end{equation}
for some $c_k \in \R$. Recall $P_{\vec{n}}(x)$ satisfies \eqref{type II orthogonality} as part of the definition. 
Since one can integrate $P_{\vec{n}}(x)$ against the orthogonal polynomial basis $\{p_k(x)\}_{k=0}^{n_j-1}$ instead of the monomial basis $\{x^k\}_{k=0}^{n_j-1}$ to check the orthogonality, we observe that the first $n_1$ orthogonality conditions, $\int_\R P_{\vec{n}}(x) p_k(x) w^{(1)}(x) \dd x = 0$ for $k = 0, \cdots, n_1 - 1$, implies that $c_0 = \dots = c_{n_1-1} = 0$.
This simplifies the expansion to
\begin{equation} \label{eq:P_expansion_pi_basis}
    P_{\vec{n}}(x) = p_n(x) + \sum_{k=n_1}^{n-1} c_k p_k(x).
\end{equation}
The remaining coefficients can be specified from the rest of the orthogonality conditions now for $\{w^{(j)}(x)\}_{j=2}^r$.
For each $j \in \{2, \dots, r\}$ and $b \in \{0, \dots, n_j-1\}$, we have
\begin{equation}
\int_{\R} P_{\vec{n}}(x) p_b(x) w^{(j)}(x) \dd x = 0.
\end{equation}
Substituting \eqref{eq:P_expansion_pi_basis} into this condition gives
\begin{equation}
\int_{\R} \left( p_n(x) + \sum_{k=n_1}^{n-1} c_k p_k(x) \right) p_b(x) w^{(j)}(x) \dd x = 0.
\end{equation}
Rearranging this gives a linear system for the coefficient vector $\bm{c}_2 = (c_{n_1}, \dots, c_{n-1})^\top$,
\begin{equation}
\sum_{k=n_1}^{n-1} c_k \left( \int_{\R} p_k(x) p_b(x) w^{(j)}(x) \dd x \right) = - \int_{\R} p_n(x) p_b(x) w^{(j)}(x) \dd x.
\end{equation}
In matrix notation, the linear system for the coefficients becomes 
\begin{equation} \label{cM=-v}
  \bm{c}_2^\top  \bm{M}_{22}=-\bm{v}_2^\top
\end{equation}
where we used \eqref{defn of v}, \eqref{defn of M}, \eqref{defn of pi phi v, 1-2}, and \eqref{defn of M, H 4 blocks}.

Since \eqref{eq:P_expansion_pi_basis} can be written as $P_{\vec{n}}(x) = p_n(x) + \bm{c}_2^\top \bm{p}_2(x) $, we obtain $ P_{\vec{n}}(x) = p_n(x) - \bm{v}_2^\top (\bm{M}_{22} )^{-1}  \bm{p}_2(x)$ from \eqref{cM=-v}. 
Moreover, by Schur's formula for the determinant of a block matrix, we have
\begin{equation}
\det
\begin{pmatrix}
\;\bm{M}_{22}  & \bm{p}_2(x) \\
 & \\
\;\bm{v}_2 ^\top & p_n(x)
\end{pmatrix} = \det\bm{M}_{22} \left( p_n(x) - \bm{v}_2^\top \bm{M}_{22}^{-1}\bm{p}_2(x)\right).
\end{equation}
Dividing by $\det \bm{M}_{22}$, we obtain the determinantal formula for $P_{\vec{n}}(x)$, which proves Theorem~\ref{typeIIreduction}.

%%%%%%%%%%%%%%%%%%%%%%%%%%%%%%%%%%%%%%%%%%

\subsection{Proof of Theorem~\ref{type I reduction}}

Similar to type II MOPRL, defining condition \eqref{type I orthogonality 1} are entirely equivalent to testing against the orthogonal polynomial basis $ \{ p_m(x) \}_{m=0}^{n-1} $,
\begin{equation} \label{type I orthogonality equiv}
    \int_{\R} Q_{\vec{n}}(x) p_m(x) \, \dd x = \delta_{m, n-1}, \quad \text{for } m = 0, \dots, n-1.
\end{equation}
We now build a linear system using these equivalent conditions. 
Since type I MOPRL $A_j(x)$ for $j=1, \cdots, r$ can be written as a linear combination of $\{p_k\}_{k=0}^{n-1}$, one has
\[
A_j(x) = \sum_{k=0}^{n_j-1} d_k^{(j)} p_k(x)
\]
for some $d_k^{(j)} \in \R$, and $Q_{\vec{n}}(x)$ is expressed by
\begin{equation} \label{Q=dTphi}
    Q_{\vec{n}}(x) = \sum_{j=1}^{r} \sum_{k=0}^{n_j-1} d_k^{(j)} p_k(x) w^{(j)}(x) = \bm{d}^\top \bm{q}(x),
\end{equation}
where $\bm{d} = \left( d_0^{(1)}, \cdots, d_{n_1 - 1}^{(1)}, d_0^{(2)}, \cdots, d_{n_2 - 1}^{(2)}, \boldsymbol{\cdots}, d_0^{(r)}, \cdots, d_{n_r - 1}^{(r)} \right)^\top $, recall \eqref{defn of phi}.
Substituting this into \eqref{type I orthogonality equiv} gives
\begin{equation}
    \bm{d}^\top \int_{\R} p_m(x) \bm{q}(x) \dd x = \delta_{m, n-1}
\end{equation}
for $m = 0, \cdots, n-1$. In matrix notation, recalling \eqref{defn of M}, this gives
\begin{align}\label{Q system}
    \bm{d}^\top \bm{M}^\top = \bm{e}_n^\top \iff \bm{M} \bm{d} = \bm{e}_n,
\end{align}
where $\bm{e}_n = (0, \cdots, 0, 1)^\top \in \R^n$. 
Thus, we have $Q_{\vec{n}}(x) = \bm{e}_n^\top \bm{M}^{\top-1}  \bm{q}(x)$, and similarly to the discussion in the previous section, we obtain the determinantal formula for $Q_{\vec{n}}(x)$, which proves Theorem~\ref{type I reduction}.

%%%%%%%%%%%%%%%%%%%%%%%%%%%%%%%%%%%%%%%%%

\subsection{Proof of Theorem~\ref{kernel reduction}}
\label{Proof of Kernel MOPRL}

We first prove a simple uniqueness lemma for the finite-dimensional kernel. This lemma also explains why the kernel can be written in the weighted basis $\bm q$.

\begin{lemma}\label{lem:kernel uniqueness}
For each fixed $x$, there is a unique function $L(x,\cdot)$ in
\begin{equation}
    \mathcal W_{\vec n}^{\mathbb R}:=\operatorname{span}\{q_1,\ldots,q_n\},
\end{equation}
where $q_1,\ldots,q_n$ are the entries of $\bm q$, such that
\begin{equation}\label{kernel reproducing property}
    \int_{\mathbb R}L(x,y)\pi(y)\,\dd y=\pi(x),  \quad \pi\in\mathcal P_{n-1}^{\mathbb R}.
\end{equation}
Moreover, the MOPRL kernel can be written as
\begin{equation}\label{K=qMpinv}
    K(x,y)=\bm q(y)^\top\bm M^{-1}\bm p(x).
\end{equation}
\end{lemma}

\begin{proof}
Suppose that $L_1(x,\cdot)$ and $L_2(x,\cdot)$ both belong to $\mathcal W_{\vec n}^{\mathbb R}$ and both satisfy \eqref{kernel reproducing property}. Then, for fixed $x$, their difference can be written as
\begin{equation}
    L_1(x,y)-L_2(x,y)=\bm a(x)^\top\bm q(y)
\end{equation}
for some $n$-column vector $\bm a(x)$. Since the difference integrates to zero against every polynomial in $\mathcal P_{n-1}^{\mathbb R}$, we may test it against the basis $p_0,\ldots,p_{n-1}$. Thus, for $m=0,\ldots,n-1$,
\begin{equation}
    0=\left(\int_{\mathbb R}p_m(y)\bm q(y)^\top\,\dd y\right)\bm a(x).
\end{equation}
Equivalently,
\begin{equation}
    \bm M\bm a(x)=\bm 0.
\end{equation}
Since the multi-index is normal, the mixed-moment matrix $\bm M$ is nonsingular. Hence $\bm a(x)=\bm 0$, and the function satisfying \eqref{kernel reproducing property} is unique.

We now check that the MOPRL kernel satisfies the same property. Recall \eqref{defn of CD kernel}; for each fixed $x$, $K(x,\cdot)$ is a linear combination of $Q_0,\ldots,Q_{n-1}$. Since the natural path remains componentwise bounded by the final multi-index $\vec n$, each $Q_j$ is a type I linear form whose polynomial coefficients have degrees bounded by $n_k-1$ in the $k$-th weight, therefore $Q_j\in\mathcal W_{\vec n}^{\mathbb R}$. Thus $K(x,\cdot)\in\mathcal W_{\vec n}^{\mathbb R}$. Moreover, since $P_0,\ldots,P_{n-1}$ are monic polynomials of degrees $0,\ldots,n-1$, they form a basis of $\mathcal P_{n-1}^{\mathbb R}$. Therefore, for any $\pi\in\mathcal P_{n-1}^{\mathbb R}$, we can write
\begin{equation}
    \pi(y)=\sum_{m=0}^{n-1}a_mP_m(y).
\end{equation}
Using the biorthogonality relation \eqref{biorth MOPRL}, we obtain
\begin{equation}
    \int_{\mathbb R}K(x,y)\pi(y)\,\dd y = \int_{\mathbb R} \left(\sum_{j=0}^{n-1}P_j(x)Q_j(y)\right)  \left(\sum_{m=0}^{n-1}a_mP_m(y)\right)\dd y
    =\sum_{j=0}^{n-1}a_jP_j(x)=\pi(x).
\end{equation}
Thus $K$ satisfies \eqref{kernel reproducing property}.

On the other hand, define
\begin{equation}
    \widehat K(x,y):=\bm q(y)^\top\bm M^{-1}\bm p(x).
\end{equation}
Clearly $\widehat K(x,\cdot)\in\mathcal W_{\vec n}^{\mathbb R}$. It is enough to test \eqref{kernel reproducing property} on the basis $p_0,\ldots,p_{n-1}$. For $m=0,\ldots,n-1$, by the definition of $\bm M$, we have
\begin{equation}
    \int_{\mathbb R}\widehat K(x,y)p_m(y)\,\dd y
    = \left(\int_{\mathbb R}p_m(y)\bm q(y)^\top\,\dd y\right)\bm M^{-1}\bm p(x)
    =\bm e_{m+1}^\top\bm M\bm M^{-1}\bm p(x) =p_m(x).
\end{equation}
Thus $\widehat K$ also satisfies \eqref{kernel reproducing property}. By uniqueness, $K=\widehat K$, which proves \eqref{K=qMpinv}.
\end{proof}

We now prove Theorem~\ref{kernel reduction}. Recall the block decompositions \eqref{defn of pi phi v, 1-2} and \eqref{M block structure}. Since $\bm M$ is nonsingular by normality, its inverse is
\begin{equation}\label{block inverse proof 33}
    \bm M^{-1}=
    \begin{pmatrix}
        \bm H_1^{-1} & -\bm H_1^{-1}\bm M_{12}\bm M_{22}^{-1}\\
        \bm 0        & \bm M_{22}^{-1}
    \end{pmatrix}.
\end{equation}
Substituting \eqref{block inverse proof 33} into \eqref{K=qMpinv} gives
\begin{equation}\label{K block proof 33}
    K(x,y)
    = \bm q_1(y)^\top\bm H_1^{-1}\bm p_1(x)+\left( \bm q_2(y)^\top -\bm q_1(y)^\top\bm H_1^{-1}\bm M_{12}  \right) \bm M_{22}^{-1}\bm p_2(x).
\end{equation}
Thus, using \eqref{K0 block}, we obtain
\begin{equation}\label{kernel difference proof 33}
    K(x,y)-K_0(x,y)
    =\left( \bm q_2(y)^\top - \bm q_1(y)^\top\bm H_1^{-1}\bm M_{12} -   w^{(1)}(y)\bm p_2(y)^\top\bm H_2^{-1}\bm M_{22} \right) \bm M_{22}^{-1}\bm p_2(x).
\end{equation}

It remains to identify the row vector in parentheses with $\bm r(y)^\top$. By definition,
\begin{equation}\label{r proof 33 start}
    \bm r(y)^\top =\bm q_2(y)^\top -  \int_{\mathbb R}K_0(s,y)\bm q_2(s)^\top\,\dd s.
\end{equation}
Using \eqref{defn of CD kernel for p}, we have
\begin{align}\label{projection proof 33}
    \int_{\mathbb R}K_0(s,y)\bm q_2(s)^\top\,\dd s
    &=w^{(1)}(y) \sum_{j=0}^{n-1}\frac{p_j(y)}{h_j} \int_{\mathbb R}p_j(s)\bm q_2(s)^\top\,\dd s \notag \\
    &= \bm q_1(y)^\top\bm H_1^{-1}\bm M_{12}   + w^{(1)}(y)\bm p_2(y)^\top\bm H_2^{-1}\bm M_{22}.
\end{align}
Substituting \eqref{projection proof 33} into \eqref{r proof 33 start}, we get
\begin{equation}\label{r proof 33 final}
    \bm r(y)^\top =\bm q_2(y)^\top -\bm q_1(y)^\top\bm H_1^{-1}\bm M_{12} - w^{(1)}(y)\bm p_2(y)^\top\bm H_2^{-1}\bm M_{22}.
\end{equation}
By \eqref{kernel difference proof 33} and \eqref{r proof 33 final},
\begin{equation}
    K(x,y)-K_0(x,y)=\bm r(y)^\top\bm M_{22}^{-1}\bm p_2(x),
\end{equation}
which completes the proof of Theorem~\ref{kernel reduction}.

%%%%%%%%%%%%%%%%%%%%%%%%%%%%%

\subsection{Proof of Corollary~\ref{cor-general-Hermite-kernel}} \label{subsec:proof-general-Hermite-kernel}

We apply the general reduction formula to the multiple Hermite weights
\begin{equation}
    w^{(1)}(x)=e^{-x^2},\quad w^{(j)}(x)=e^{a_jx-x^2},\quad j=2,\cdots,r.
\end{equation}
It is a well-known fact that the monic orthogonal polynomials $p_k(x)$ with respect to the weight $w^{(1)}(x) = e^{-x^2}$ are the monic Hermite polynomials, 
\begin{align} \label{defn of H-hat}
\widehat{H}_k(x) = \frac{(-1)^k}{2^k} e^{x^2} \frac{\dd^k}{\dd x^k} e^{-x^2}.
\end{align}
Throughout this subsection, we use several facts about Hermite polynomials, see for example, \cite[Sections 18.3, 18.5, 18.17, 18.18]{NIST:DLMF}.

\subsubsection{Explicit Expression of $\boldsymbol{M}$}

Recall that the mixed-moment matrix $\bm M$ admits the block structures \eqref{M=(M1 M2 ... Mr)} and  \eqref{M block structure}, i.e.,
\[
M^{(1)} = \begin{pmatrix}
    \bm H_1\\
    \bm 0
\end{pmatrix} \in \operatorname{Mat}_{n\times n_1}, \quad \begin{pmatrix}
    M^{(2)} & \cdots & M^{(r)}
\end{pmatrix} = \begin{pmatrix}
    \bm M_{12}\\
    \bm M_{22}
\end{pmatrix} \in \operatorname{Mat}_{n\times (n - n_1)}.
\]
By \eqref{defn of H}, $\bm{H}_1$ is a diagonal matrix of norms $h_k$ for $k=0, \cdots, n_1-1$, where each $h_k$ is given by
\begin{align} \label{app hk}
h_k = \int_\R \widehat{H}_k^2(x) e^{-x^2} \, \dd x = \frac{\sqrt{\pi} \, k!}{2^k}.
\end{align}
By \eqref{defn of M}, each $n \times n_j$ sub-block matrix of the right $n \times (n-n_1)$ block matrix of $\bm M$,
\[
\big(\underbrace{M^{(2)}}_{n \times n_2} \; \cdots \; \underbrace{M^{(r)}}_{n \times n_r} \big),
\]
has entries
\begin{align} \label{defn of M^{(j)}_{m,b}}
M^{(j)}_{m,b} = \int_{\mathbb R}\widehat H_m(x)\widehat H_b(x)e^{a_jx-x^2}\,\dd x, \quad 0 \leq m \leq n-1, \quad 0 \leq b \leq n_j-1.
\end{align}
Since it holds that
\[
\widehat H_m(x)\widehat H_b(x)=\sum_{q=0}^{\min(m,b)}2^{-q}q!\binom{m}{q}\binom{b}{q}\widehat H_{m+b-2q}(x)
\]
and
\[
\int_{\mathbb R}\widehat H_k(x)e^{a_jx-x^2}\,\dd x=\sqrt{\pi}\,e^{t_j^2}t_j^k, \quad t_j = \frac{a_j}{2},
\]
it follows that
\begin{equation}\label{eq:G-derivative-proof}
M^{(j)}_{m,b} = \sqrt{\pi}\,e^{t_j^2}\sum_{q=0}^{\min(m,b)}2^{-q}q!\binom{m}{q}\binom{b}{q}t_j^{m+b-2q} = \sqrt{\pi}\,e^{t_j^2} \sum_{q=0}^{\min(m,b)}2^{-q} \binom{b}{q} t_j^{b-q} \left[\frac{\dd^q}{\dd t^q}t^{n_1+s}\right]_{t=t_j}.
\end{equation}

\subsubsection{Decomposition of $\boldsymbol{M}_{22}$}

Since we need to compute $\bm M_{22}^{-1}$ explicitly to apply Theorem~\ref{kernel reduction}, we shall further study structure of $\bm M_{22}$.
Let $N = n-n_1$. Then, one may write $\bm M_{22}$ as
\[
\big(\underbrace{\bm M_{22}^{(2)}}_{N \times n_2} \; \cdots \; \underbrace{\bm M_{22}^{(r)}}_{N \times n_r} \big),
\]
where the $(s,b)$-entry of $\bm M_{22}^{(j)}$ is given by
\begin{align} \label{eq:U-proof-M}
\left(\bm M_{22}^{(j)}\right)_{s,b} = M^{(j)}_{n_1+s,b}, \quad 0 \leq s \leq N-1, \quad 0 \leq b \leq n_j-1.
\end{align}
One can readily notice that \eqref{eq:G-derivative-proof} gives the matrix decomposition,
\[
\bm M_{22} = \bm V \bm R,
\]
where $\bm V$ is an $N \times N$ matrix,
\begin{equation} \label{cor 4.3 proof V}
    \bm V = \big(\underbrace{\bm V^{(2)}}_{N \times n_2} \; \cdots \; \underbrace{\bm V^{(r)}}_{N \times n_r} \big), \quad
    \bm V^{(j)}_{s,q} = \left[\frac{\dd^q}{\dd t^q}t^{n_1+s}\right]_{t=t_j}, \quad 0 \leq s \leq N-1, \quad 0 \leq q \leq n_j-1,
\end{equation}
and $\bm R$ is an $N \times N$ block diagonal matrix,
\begin{equation}
    \bm R=\operatorname{diag}(\underbrace{\bm R^{(2)}}_{n_2 \times n_2},\cdots,\underbrace{\bm R^{(r)}}_{n_r \times n_r}), \quad \bm R^{(j)}_{q,b} = \sqrt{\pi}\,e^{t_j^2}2^{-q}\binom{b}{q}t_j^{b-q},\quad 0\leq q,b\leq n_j-1,
\end{equation}
with a convention $\binom{b}{q}=0$ if $q>b$ (that is each $\bm R^{(j)}$ is upper triangular). Thus, all we need to compute is $\bm R^{-1}$ and $\bm V^{-1}$.

\subsubsection{Inverse of $\boldsymbol{R}$}

For $\bm R^{-1}$, since the inverse of a block diagonal matrix is the matrix formed by the inverses of each individual block matrix, we need to compute $\big( \bm R^{(j)} \big)^{-1}$.
Notice one can write $\bm R^{(j)}$ as a product of matrices,
\[
\bm R^{(j)} = \sqrt{\pi} e^{t_j^2} \operatorname{diag}(2^{-b})_{b=0}^{n_j-1} \, S(t_j),
\]
where $S(t_j)$ is a shifted generalized Pascal matrix with entries
\[
(S(t_j))_{b,q} = \begin{dcases}
\binom{q}{b} t_j^{q-b}, & b \leq q,\\
0, & b > q,
\end{dcases}
\]
whose inverse formula is well-known, see \cite[Equation (7)]{Pascal} for example,
\[
S(t_j)^{-1} = S(-t_j).
\]
Thus, again with the convention $\binom{q}{b}=0$ if $b>q$, one has
\begin{equation}\label{eq:Rinv-proof}
    \left(\bm R^{(j)}\right)^{-1}_{b,q} = \frac{e^{-t_j^2}}{\sqrt{\pi}} \binom{q}{b} (-t_j)^{q-b} 2^q.
\end{equation}

\subsubsection{Inverse of $\boldsymbol{V}$}

To compute $\bm V^{-1}$, we first notice that $\bm V$ is a confluent Vandermonde matrix. By computing the associated generalized Hermite interpolation polynomials, see \eqref{eq:Pi-proof}, we shall find $\bm V^{-1}$. For the general construction of such polynomials, see \cite{Spitzbart1960}.
For $j=2,\ldots,r$, define
\begin{equation}
    L_j(z)=\left(\frac{z}{t_j}\right)^{n_1}\prod_{\substack{i=2\\ i\neq j}}^r\left(\frac{z-t_i}{t_j-t_i}\right)^{n_i},
\end{equation}
the Lagrange-type factor associated with the node $t_j$. Using this factor, introduce, for $0\leq q\leq n_j-1$, the interpolation polynomials
\begin{equation}\label{eq:Pi-proof}
    \Pi_{j,q}(z) = L_j(z) \, \frac{(z-t_j)^q}{q!} \left( \sum_{\mu=0}^{n_j-1-q}\frac{1}{\mu!}\left[\frac{\dd^\mu}{\dd z^\mu}\frac{1}{L_j(z)}\right]_{z=t_j}(z-t_j)^\mu \right),
\end{equation}
where the expression in parentheses is the Taylor polynomial of $1/L_j(z)$ at $t_j$ of degree $n_j-1-q$. The polynomial $\Pi_{j,q}$ satisfies the interpolation conditions
\begin{equation}\label{eq:fund}
    \left[\frac{\dd^\nu}{\dd z^\nu}\Pi_{j,q}(z)\right]_{z=t_i} =\delta_{ij}\delta_{q\nu},\qquad 2\leq i,j\leq r,\qquad 0\leq \nu\leq n_i-1.
\end{equation}
In the special case $n_j=1$, only $q=0$ occurs at the node $t_j$, and the construction reduces to $\Pi_{j,0}=L_j$.

Let us verify that condition \eqref{eq:fund} indeed holds. First, observe that $L_j(z)$ has a zero of order $n_i$ at $z=t_i$ for $i \neq j$, while the remaining factors of $\Pi_{j,q}(z)$ are analytic at $z=t_i$. Hence, $\Pi_{j,q}(z)$ has a zero of order $n_i$ at $z = t_j$ for $i \neq j$, establishing \eqref{eq:fund} for $i \neq j$. For the $i=j$ case, Taylor's theorem gives 
\begin{equation}
    \frac{1}{L_j(z)} - \left(\sum_{\mu=0}^{n_j-1-q}\frac{1}{\mu!}\left[\frac{\dd^\mu}{\dd z^\mu}\frac{1}{L_j(z)}\right]_{z=t_j}(z-t_j)^\mu \right) = \mathcal O \left((z-t_j)^{n_j-q}\right)
\end{equation}
as $z \to t_j$, and combining this with \eqref{eq:Pi-proof}, we have
\begin{equation}
    \Pi_{j,q}(z)=\frac{(z-t_j)^q}{q!}+O\bigl((z-t_j)^{n_j}\bigr).
\end{equation}
This proves \eqref{eq:fund} at the node $t_j$.

Now, we shall rewrite $\Pi_{j,q}(z)$ as
\[
\Pi_{j,q}(z) = \sum_{s=0}^{N-1} v_s^{(j,q)} z^{n_1+s}
\]
where the coefficients $v_s^{(j,q)}$ can be computed from expanding \eqref{eq:fund} at $z=0$. 
Using this expression with \eqref{eq:fund}, one has
\begin{equation} \label{cor 4.3 proof V V^{-1} = I}
    \sum_{s=0}^{N-1}v_s^{(j,q)} \left[\frac{\dd^\nu}{\dd z^\nu}z^{n_1+s}\right]_{z=t_i} = \sum_{s=0}^{N-1}v_s^{(j,q)} \bm V^{(i)}_{s,\nu} =  \delta_{ij}\delta_{q\nu},
\end{equation}
where we also recall \eqref{cor 4.3 proof V}.
Equation \eqref{cor 4.3 proof V V^{-1} = I} implies that we can form $n_j \times N$ block matrices $\bm W^{(j)}$ where the $(q,s)$-entry of $j$-th block is given by $v_s^{(j,q)}$ and put these block matrices together to construct $\bm V^{-1}$,
\begin{equation} \label{cor 4.3 proof V^{-1} final}
    \bm V^{-1}=\begin{pmatrix}\bm W^{(2)}\\ \vdots\\ \bm W^{(r)}\end{pmatrix}, \quad \bm W^{(j)}_{q,s} = v_s^{(j,q)}, \quad 0 \leq q \leq n_j-1, \quad 0 \leq s \leq N-1.
\end{equation}

\subsubsection{Explicit Expression of $\boldsymbol{M}_{22}^{-1}$}

Recall $\bm M_{22}^{-1} = \bm R^{-1} \bm V^{-1}$. We shall write
\begin{equation}\label{eq:U-proof-M-inv}
    \bm M_{22}^{-1}=\begin{pmatrix}\bm U^{(2)}\\ \vdots\\ \bm U^{(r)}\end{pmatrix}, \quad
    \bm U^{(j)}_{b,s}=U^{(j)}_{b,s}, \quad 0\leq b\leq n_j-1, \quad 0\leq s\leq N-1.
\end{equation}
Applying \eqref{eq:Rinv-proof} and \eqref{cor 4.3 proof V^{-1} final}, we obtain $\bm U^{(j)} = \bm R^{(j)-1} \bm W^{(j)}$, or entry-wise,
\begin{equation}\label{eq:U-proof}
    U^{(j)}_{b,s}=\frac{e^{-t_j^2}}{\sqrt{\pi}}\sum_{q=b}^{n_j-1}(-1)^{q-b}2^q t_j^{q-b}\binom{q}{b}v_s^{(j,q)},
\end{equation}
which is the coefficient used in \eqref{eq:U-Gamma-explicit}.

\subsubsection{Explicit Expression of $K(x,y)$}

Now, we are ready to compute the multiple Hermite kernel $K(x,y)$. To this end, we first apply the considered setup to \eqref{defn of CD kernel for p} to obtain the standard Hermite kernel
\begin{equation} \label{eq:K0-proof-3.4}
K_0(x,y) = e^{-y^2}\sum_{m=0}^{n-1}\frac{\widehat H_m(x)\widehat H_m(y)}{h_m} = e^{-y^2}\sum_{m=0}^{n_1-1}\frac{\widehat H_m(x)\widehat H_m(y)}{h_m} + e^{-y^2}\sum_{t=0}^{N-1}\frac{\widehat H_{n_1+t}(x)\widehat H_{n_1+t}(y)}{h_{n_1+t}},
\end{equation}
where the last expression will be useful later.
Similarly, the current setup shows that the vector $\bm p_2(x)$, recall \eqref{defn of pi} and \eqref{defn of pi phi v, 1-2}, consists of the components
\begin{equation}\label{eq:residual-proof-p}
    p_s(x) = \widehat H_{n_1 + s}(x), \quad 0 \leq s \leq N-1
\end{equation}
the vector $\bm q_2(y)$, recall \eqref{defn of phi} and \eqref{defn of pi phi v, 1-2}, consists of the components
\begin{equation}
    \sigma_{j,b}(y) = \widehat H_b(y)e^{a_jy-y^2}, \quad 2 \leq j \leq r, \quad 0\leq b\leq n_j-1
\end{equation}
and the vector $\bm r(y)$, recall \eqref{defn of q}, consists of the components
\begin{equation}\label{eq:residual-proof}
    r_{j,b}(y)=\sigma_{j,b}(y)-\int_{\mathbb R}K_0(x,y) \sigma_{j,b}(x)\,\dd x
    = \widehat H_b(y)e^{a_jy-y^2}-e^{-y^2}\sum_{m=0}^{n-1}\frac{M^{(j)}_{m,b}}{h_m}\widehat H_m(y), 
\end{equation}
for $2 \leq j \leq r$ and $0\leq b\leq n_j-1$, where we used \eqref{defn of M^{(j)}_{m,b}} for the last equality. 

Applying \eqref{eq:residual-proof-p}, \eqref{eq:U-proof-M-inv}, and \eqref{eq:residual-proof} to Theorem~\ref{kernel reduction}, one has
\begin{multline}\label{eq:kernel-reduction-proof}
    K(x,y)-K_0(x,y) = \sum_{s=0}^{N-1}\sum_{j=2}^r\sum_{b=0}^{n_j-1}U^{(j)}_{b,s}\widehat H_{n_1+s}(x)\widehat H_b(y)e^{a_jy-y^2}\\
    -e^{-y^2}\sum_{s=0}^{N-1}\sum_{m=0}^{n-1}\frac{\widehat H_{n_1+s}(x)\widehat H_m(y)}{h_m}
    \sum_{j=2}^r\sum_{b=0}^{n_j-1}U^{(j)}_{b,s}M^{(j)}_{m,b}.
\end{multline}
We shall denote the inner sum on the last line by
\begin{align} \label{defn of Gamma_ms}
\Gamma_{m,s} := \sum_{j=2}^r\sum_{b=0}^{n_j-1}U^{(j)}_{b,s}M^{(j)}_{m,b}.
\end{align}
For $m=n_1+t$ where $0\leq t\leq N-1$, one has
\begin{equation}
\Gamma_{m,s} = \delta_{st}, \label{Gamma_ms 1}
\end{equation}
which is readily obtained from writing out $\bm M_{22}^{-1}\bm M_{22} = I$, recall \eqref{eq:U-proof-M} and \eqref{eq:U-proof-M-inv}. For $0\leq m\leq n_1-1$, one can apply \eqref{eq:U-proof} and \eqref{eq:G-derivative-proof} to the definition \eqref{defn of Gamma_ms} of $\Gamma_{m,s}$ and readily show that
\begin{equation}
    \Gamma_{m,s} = \sum_{j=2}^r\sum_{q=0}^{n_j-1}\left[\frac{\dd^q}{\dd t^q}t^m\right]_{t=t_j}v_s^{(j,q)} = \sum_{j=2}^r\sum_{q=0}^{\min(m,n_j-1)}q!\binom{m}{q}t_j^{m-q}v_s^{(j,q)}. \label{Gamma_ms 2}
\end{equation}

Using \eqref{Gamma_ms 1}, \eqref{Gamma_ms 2}, and \eqref{eq:K0-proof-3.4}, it follows from \eqref{eq:kernel-reduction-proof} that
\begin{align*}
    K(x,y) &= K_0(x, y) + \sum_{s=0}^{N-1}\sum_{j=2}^r\sum_{b=0}^{n_j-1}U^{(j)}_{b,s}\widehat H_{n_1+s}(x)\widehat H_b(y)e^{a_jy-y^2}\\
    &\qquad \qquad \quad -e^{-y^2}\sum_{s=0}^{N-1} \left(\sum_{m=0}^{n_1-1}\frac{\widehat H_{n_1+s}(x)\widehat H_m(y)}{h_m}
    \Gamma_{m,s} + \sum_{t=0}^{N-1}\frac{\widehat H_{n_1+s}(x)\widehat H_{n_1+t}(y)}{h_{n_1+t}}
    \delta_{st} \right)\\
    &= e^{-y^2}\sum_{m=0}^{n_1-1}\frac{\widehat H_m(x)\widehat H_m(y)}{h_m} + e^{-y^2}\sum_{t=0}^{N-1}\frac{\widehat H_{n_1+t}(x)\widehat H_{n_1+t}(y)}{h_{n_1+t}}\\
    &\qquad \qquad \quad + \sum_{s=0}^{N-1}\sum_{j=2}^r\sum_{b=0}^{n_j-1}U^{(j)}_{b,s}\widehat H_{n_1+s}(x)\widehat H_b(y)e^{a_jy-y^2}\\
    &\qquad \qquad \quad -e^{-y^2}\sum_{s=0}^{N-1} \sum_{m=0}^{n_1-1}\frac{\widehat H_{n_1+s}(x)\widehat H_m(y)}{h_m} \Gamma_{m,s} - e^{-y^2} \sum_{t=0}^{N-1}\frac{\widehat H_{n_1+t}(x)\widehat H_{n_1+t}(y)}{h_{n_1+t}}\\
    &=e^{-y^2}\sum_{m=0}^{n_1-1}\frac{\widehat H_m(y)}{h_m}\left(\widehat H_m(x)-\sum_{s=0}^{N-1}\Gamma_{m,s}\widehat H_{n_1+s}(x)\right)
    +\sum_{s=0}^{N-1}\sum_{j=2}^r\sum_{b=0}^{n_j-1}U^{(j)}_{b,s}\widehat H_{n_1+s}(x)\widehat H_b(y)e^{a_jy-y^2}.
\end{align*}
Using \eqref{app hk}, we come to the desired result.
 
\section{Proof of Theorems (MOPUC)} \label{appendix mopuc}

This section is in parallel with Section~\ref{section proofs} (for MOPRL), but here our focus is on MOPUC.
 
\subsection{Proof of Theorem~\ref{thm:typeII}} \label{sec: thm 1}

Since OPUC is a basis of $\mathcal{P}^{\TT}_n$, one can write $\Phi_{\vec{n}}(z)$ as a linear combination of them, i.e.,
\begin{equation}\label{eq:typeII-ansatz}
\Phi_{\vec{n}}(z) = \varphi_n(z) + \sum_{j=0}^{n-1} \gamma_j\,\varphi_j(z), \quad \gamma_j \in \C,
\end{equation}
where we recall that $\Phi_{\vec{n}}(z)$ is a monic polynomial.
By \eqref{eq:typeII MOPUC-orth} and Remark~\ref{rem 1}, it holds for $k=0,1, \cdots, n_1-1$ that
\begin{equation} \label{orthogonality proof thm 1}
\int_{\TT} \Phi_{\vec{n}}(z) \, \overline{\varphi_k(z)} \, \omega^{(1)}(z) \frac{|\dd z|}{2\pi} = 0.
\end{equation}
Substituting \eqref{eq:typeII-ansatz} into the above orthogonality condition, we get
\begin{equation}
\text{LHS of \eqref{orthogonality proof thm 1}} = \int_{\TT} \varphi_n(z) \, \overline{\varphi_k(z)} \, \omega^{(1)}(z) \frac{|\dd z|}{2 \pi} +  \sum_{j=0}^{n-1} \gamma_j \int_{\TT} \varphi_j(z) \, \overline{\varphi_k (z)} \, \omega^{(1)} \frac{|\dd z|}{2\pi} = \gamma_k h_k,
\end{equation}
for $k = 0, 1, \cdots, n_1 - 1$, but since the RHS of \eqref{orthogonality proof thm 1} is always zero, we have
$\gamma_0 = \cdots = \gamma_{n_1-1} = 0$. Hence,
\begin{equation} \label{eq:typeII-ansatz2}
\Phi_{\vec{n}}(z) = \varphi_n(z) + \sum_{k = n_1}^{n-1} \gamma_k \, \varphi_k(z) = \varphi_n (z) + \bm{\gamma}_2^\top \, \bm{\varphi}_2(z),
\end{equation}
where $\bm \gamma_2 = (\gamma_{n_1}, \cdots, \gamma_{n-1})^\top$.
By integrating \eqref{eq:typeII-ansatz2} against $\bm{\psi}_2(z)$ with respect to the arclength measure, where we use \eqref{eq:typeII MOPUC-orth} and Remark~\ref{rem 1}, one has
\[
\bm{0} = \int_{\TT} \left( \varphi_n (z) + \bm \gamma_2^\top  \bm{\varphi}_2(z) \right) \overline{\bm{\psi}_2(z)}^\top \frac{|\dd z|}{2\pi} 
\iff \bm{0} = \bm{u}^\top + \bm{\gamma}_2^\top \bm{\mathcal{M}}_{22}
\iff \bm{\gamma}_2^\top = -\bm{u}^\top \bm{\mathcal{M}}_{22}^{-1},
\]
which proves Theorem~\ref{thm:typeII}.

%%%%%%%%%%%%%%%%%%%%%%%%%%%%%

\subsection{Proof of Theorem~\ref{thm:typeI}} \label{sec: thm 2}

Since type I MOPUC $\Lambda_j(z)$ for $j=0, \cdots, r$ can be written as a linear combination of \(\{\varphi_k(z)\}_{k=0}^{n-1}\), one has
\[
\Lambda_j(z) = \sum_{k=0}^{n_j-1} \delta_k^{(j)} \varphi_k(z), \quad \delta_k^{(j)} \in \C,
\]
and the type I linear form $\Psi_{\vec{n}}(z)$ is expressed by
\begin{equation} \label{Psi = delta^T psi}
    \Psi_{\vec{n}}(z) = \sum_{j=1}^{r} \sum_{k=0}^{n_j-1} \delta_k^{(j)} \varphi_k(z) \omega^{(j)}(z) = \bm{\delta}^\top \bm{\psi}(z),
\end{equation}
where \(\bm{\delta} = \left( \delta_0^{(1)}, \cdots, \delta_{n_1 - 1}^{(1)}, \delta_0^{(2)}, \cdots, \delta_{n_2 - 1}^{(2)}, \boldsymbol{\cdots}, \delta_0^{(r)}, \cdots, \delta_{n_r - 1}^{(r)} \right)^\top \).
By \eqref{eq:typeI MOPC-orth} and Remark~\ref{rem 1}, we have
\begin{equation}
    \overline{\bm{\delta}}^\top \int_{\TT} \varphi_m(z) \, \overline{\bm{\psi}(z)} \, \frac{|\dd z|}{2 \pi} = \delta_{m, n-1}
\end{equation}
for $m = 0, \cdots, n-1$. 
This gives the linear system
\begin{equation}\label{eq:Md-e}
\overline{\bm{\delta}}^\top \bm{\mathcal{M}}^\top  = \bm{e}_n^\top \iff
\bm{\delta}^\top = \bm{e}_n^\top \overline{\bm{\mathcal{M}}}^{\top-1}.
\end{equation}
Substituting this back into \eqref{Psi = delta^T psi} completes the proof of Theorem~\ref{thm:typeI}.

%%%%%%%%%%%%%%%%%%%%%%%%%%%%%%%%%%

%%%%%%%%%%%%%%%%%%%%%

\subsection{Proof of Theorem~\ref{thm:kernel}}

 We consider the natural path from $\vec{n}_0=\vec{0}$ to $\vec{n}_n=\vec n$ and the associated
families of type II MOPUC and type I linear forms
\begin{equation}
    \Phi_j(z)=\Phi_{\vec n_j}(z),     \qquad \Psi_j(\xi)=\Psi_{\vec n_{j+1}}(\xi),  \qquad j=0,\ldots,n-1.
\end{equation}
Recall that the MOPUC Christoffel-Darboux kernel is defined by
\begin{equation}
    \mathcal K(z,\xi)=\sum_{j=0}^{n-1}\Phi_j(z)\overline{\Psi_j(\xi)}.
\end{equation}
The proof of Theorem~\ref{thm:kernel} is the same as the proof of Theorem~\ref{kernel reduction} in the MOPRL case, see
Section~\ref{Proof of Kernel MOPRL}, after replacing
\begin{equation}
    \bm p,\quad \bm q,\quad \bm M,\quad K,\quad K_0,\quad \bm r
\end{equation}
respectively by
\begin{equation}
    \bm\varphi,\quad \overline{\bm\psi},\quad \bm{\mathcal M},\quad \mathcal K,\quad \mathcal K_0,\quad \bm\pi.
\end{equation}
The analogue of Lemma~\ref{lem:kernel uniqueness} is the following.

\begin{lemma}\label{lem:MOPUC-kernel-uniqueness}
For each fixed $z$, there is a unique function $L(z,\cdot)$ in
\begin{equation}\label{defn of Wn MOPUC}
    \mathcal W_{\vec n}^{\mathbb T}:=
    \operatorname{span}\{\overline{\psi_1},\cdots,\overline{\psi_n}\},
\end{equation}
where $\psi_1,\cdots,\psi_n$ are the entries of $\bm{\psi}$, such that
\[
\int_{\mathbb T} L(z,\xi)\varphi(\xi)\frac{|\dd \xi|}{2\pi} = \varphi(z), \quad \varphi \in \mathcal P_{n-1}^{\TT}.
\]
Moreover, the MOPUC Christoffel-Darboux kernel admits the representation
\begin{equation}\label{K=psiMphiinv MOPUC}
    \mathcal K(z,\xi)
    = \overline{\bm{\psi}(\xi)}^\top  \bm{\mathcal M}^{-1}  \bm{\varphi}(z).
\end{equation}
\end{lemma}

Indeed, by Lemma~\ref{lem:MOPUC-kernel-uniqueness} and the same block-inverse computation as in
Theorem~\ref{kernel reduction}, we obtain
\begin{equation}
    \mathcal K(z,\xi)-\mathcal K_0(z,\xi)  = \bm\pi(\xi)^\top  \bm{\mathcal M}_{22}^{-1} \bm\varphi_2(z),
\end{equation}
where
\begin{equation}
    \bm\pi(\xi)^\top = \overline{\bm{\psi}_2(\xi)}^\top -\int_{\mathbb T}\mathcal K_0(s,\xi) \overline{\bm{\psi}_2(s)}^\top \frac{|\dd s|}{2\pi}.
\end{equation}
This proves Theorem~\ref{thm:kernel}.
\section{Proof of Corollaries \ref{cor: MOPUC type II eg 1} -- \ref{cor: MOPUC kernel eg 2}}
\label{section example MOPUC}

In this section, we consider the MOPUC $\Phi_{\vec{n}}(z)$ and $\Psi_{\vec{n}}(z)$ defined by the multi-index 
$\vec{n} = (n-r, 1, \cdots, 1) \in \N^{r+1}$ and the weights of orthogonality on $\TT$,
\begin{equation} \label{eg 1: weights}
\omega^{(1)}(z) := 1,\quad
\omega^{(j+1)}(z) := \mathbf{1}_{E_j}(\arg z),\quad
E_j=[\alpha_j,\beta_j] \subseteq [0, 2\pi), \quad j=1,\cdots,r,
\end{equation}
where we assumed the interval of angles $E_1, \cdots, E_r$ are pairwise disjoint and has the equal length $\Delta$.
Before we proceed with the proof, we shall introduce some notation.
For $E_j = [\alpha_j,\beta_j] \subseteq [0,2\pi)$ and $\ell \in \Z$, we define
\begin{equation}\label{defn of mu}
\mu_{E_j}(\ell) :=\int_{\TT} z^\ell \,\mathbf{1}_{E_j} (\arg z) \, \frac{|\dd z|}{2\pi} = 
\begin{dcases}
a_0, & \ell=0,\\
a_{\ell} \, x_j^\ell, & \ell\neq 0,
\end{dcases}
\end{equation}
where we recall \eqref{defn of a_l} and \eqref{defn of x_j} for $a_\ell$ and $x_j$. We note that the simple relation, $\overline{\mu_{E_j}(\ell)} = \mu_{E_j}(-\ell)$.

%%%%%%%%%%%%%%%%%%%%%%%%%%%%%%%%%%%%%%%%

\subsection{Proof of Corollary~\ref{cor: MOPUC type II eg 1}}

If the weight of orthogonality is given by $\omega^{(1)}(z) = 1$ on $\TT$, it is an easy exercise to show that the corresponding OPUC of degree $n$ is given by
\begin{align} \label{eg 1: varphi_n}
\varphi_n (z) = z^n,
\end{align}
the norm $h_n$ of $\varphi_n (z)$ is $1$, and the Christoffel-Darboux kernel is given by
\begin{align} \label{eg 1 K0}
\mathcal{K}_0(z, \xi) = \sum_{j=0}^{n-1} \left( z \Bar{\xi} \right)^j = \frac{1 - \left( z \Bar{\xi} \right)^n}{1 - z \Bar{\xi}}.
\end{align}
Applying \eqref{eg 1: varphi_n} and \eqref{eg 1: weights} into \eqref{eq:pi-primary} and \eqref{eq:chi-stacked}, we have
\begin{align}\label{eg 1: varphi, psi}
\begin{aligned}
\boldsymbol{\varphi}(z) &= \big( 1, z,\cdots, z^{n-1} \big)^\top,\\
\boldsymbol{\psi}(z) &= \big( 1, z,\cdots, z^{n-r-1}, \mathbf{1}_{E_1}(\arg z), \mathbf{1}_{E_2}(\arg z), \cdots, \mathbf{1}_{E_r}(\arg z) \big)^\top.
\end{aligned}
\end{align}
By \eqref{eq:v-def}, one can directly compute $\boldsymbol{u}$ to be
\[
\boldsymbol{u} = ( \underbrace{0, \cdots, 0}_{n-r}, \mu_{E_1}(n), \mu_{E_2}(n), \cdots, \mu_{E_r}(n) )^\top
\]
using \eqref{defn of mu}.
Based on our convention, we denote the first $n-r$ entries of the above vectors with subindex $1$ and the remaining $r$ entries with subindex $2$.
Similarly to the above computations, one can explicitly compute the mixed-moment matrix
\[
\bm{\mathcal{M}} = \big( \underbrace{\mathcal{M}^{(1)}}_{n \times (n-r)} \;  \underbrace{\mathcal{M}^{(2)}}_{n \times 1} \;  \cdots \; \underbrace{\mathcal{M}^{(r+1)}}_{n \times 1} \big).
\]
Since the $(a,b)$-entry of $k$-th block matrix in $\bm{\mathcal M}$ is given by
\begin{align} \label{eg 1 M block entry}
\mathcal{M}^{(k)}_{a, b} = \begin{dcases}
    \begin{dcases}
        1, & a = b,\\
        0, & \text{else},
    \end{dcases}, & k=1, \\
    \mu_{E_{k-1}}(a-b), & k \geq 2,
\end{dcases}
\end{align}
one can find the $(q, p)$-entry of $\bm{\mathcal{M}}_{22}$, see also \eqref{eg 1 M22}, as 
\begin{align} \label{q,p entry of M22}
\left( \bm{\mathcal{M}}_{22} \right)_{q,p} = a_{l_q} x_p^{l_q}, \quad 1 \leq p, q \leq r.
\end{align}
By Theorem~\ref{thm:typeII}, the considered type II MOPUC is expressed as
\begin{align} \label{eg 1 expression 1}
\Phi_{\vec{n}}(z) = z^n - \bm{\gamma}^\top \bm{\varphi}_2(z)
\end{align}
where $\bm \gamma = (\gamma_1, \gamma_2 \cdots, \gamma_r)$ is an $r$-column vector such that 
\begin{align} \label{gamma M22 = u2}
\bm \gamma^\top \bm{\mathcal{M}}_{22} = \boldsymbol{u}_2^\top.
\end{align}
To solve \eqref{gamma M22 = u2} for $\bm \gamma^\top$ effectively, we consider two polynomials of degree $r-1$,
\[
Y(x) = \sum_{q=1}^r a_{l_q} \gamma_q x^{q-1} \quad \text{and} \quad Z(x) = a_n \left( x^r - \prod_{m=1}^r (x - x_m) \right).
\]
By \eqref{q,p entry of M22} and \eqref{gamma M22 = u2}, one can readily verify that
\[
Y(x_p) = Z(x_p) = a_n x_p^r, \quad p = 1, 2, \cdots, r.
\]
Since degree $r-1$ polynomials that give the same values at $r$ points must be identically the same, we have
\begin{align}
    a_n \left( x^r - \prod_{m=1}^r (x - x_m) \right) = \sum_{q=1}^r a_{l_q} \gamma_q \, x^{q-1} \iff \sum_{s=0}^{r-1} a_n (-1)^{r-s+1} e_{r-s}(x_1, \cdots, x_r) x^s = \sum_{q=1}^r a_{l_q} \gamma_q \, x^{q-1}
\end{align}
where the last expression is due to Vieta's formulas. This gives the formula,
\begin{align} \label{gamma q formula}
    \gamma_q = \frac{a_n}{a_{l_q}}(-1)^{r-q} e_{r-q+1}(x_1, \cdots, x_r), \quad 1 \leq q \leq r
\end{align}
and together with \eqref{eg 1 expression 1} and \eqref{eg 1: varphi, psi}, we finish the proof of \eqref{MOPUC type II eg 1}.

We can rewrite \eqref{MOPUC type II eg 1} as
\[
\Phi_{\vec{n}}(z) = z^{n-r} \left( z^r - \sum_{q=1}^r \gamma_q z^{q-1} \right),
\]
therefore, there exist a zero of multiplicity $n-r$ at the origin and $r$ zeros at $\zeta_1, \zeta_2, \cdots, \zeta_r \in \C$ where they are solutions of $R(z) = 0$ with
\[
R(z) := z^r - \sum_{q=1}^r \gamma_q z^{q-1}.
\]
Since $R(z)$ admits the unique factorization, combining with Vieta's formulas, one gets
\[
R(z) = \prod_{j=1}^r (z - \zeta_j) = z^r + \sum_{s = 0}^{r-1} (-1)^{r-s} e_{r-s}(\zeta_1, \cdots, \zeta_r) z^s.
\]
Thus, it holds for $q = 1, 2, \cdots, r$ that
\[
-\gamma_{q} = (-1)^{r-q+1} e_{r-q+1} (\zeta_1, \cdots, \zeta_r) \implies |\gamma_q| = |e_{r-q+1} (\zeta_1, \cdots, \zeta_r)|.
\]
By \eqref{gamma q formula}, we get
\[
|e_{r-q+1} (\zeta_1, \cdots, \zeta_r)| = \left| \frac{a_n}{a_{l_q}} \right| |e_{r-q+1} (x_1, \cdots, x_r)|
\]
By the definition of $e_s(\zeta_1, \cdots, \zeta_r)$, see \eqref{defn of esp}, one has trivially
\[
|e_s(\zeta_1, \cdots, \zeta_r)| \leq \sum_{1 \leq i_1 < i_2 < \cdots < i_s \leq r} |\zeta_{i_1}| |\zeta_{i_2}| \cdots |\zeta_{i_s}| \leq \binom{r}{s} \left( \max_{1\leq j \leq r} |\zeta_j| \right)^s.
\]
Combining the last two displays, we obtain \eqref{eg 1 escaping} and finish the proof of this corollary.

%%%%%%%%%%%%%%%%%%%%%%%%%%%%%%%%%%%%%%%%%

\subsection{Proof of Corollary~\ref{cor: MOPUC type I eg 1}}

By \eqref{eg 1 M block entry}, one can write $\bm{\mathcal{M}}$ as
\[
\bm{\mathcal{M}} = \begin{pmatrix}
\begin{array}{c|c}
    \bm{I}      & \bm{\mathcal{M}}_{12} \\
    \noalign{\vskip 2pt\hrule\vskip 2pt}
    \bm{0} & \bm{\mathcal{M}}_{22}
\end{array}
\end{pmatrix}
\]
where $\bm I$ is the $(n-r)\times(n-r)$ identity matrix, $\bm{0}$ is the $r \times (n-r)$ zero matrix,
\[
\bm{\mathcal{M}}_{12} = \begin{pmatrix}
        \mu_{E_{1}}(0)     & \mu_{E_{2}}(0)     & \cdots & \mu_{E_{r}}(0)     \\
        \mu_{E_{1}}(1)     & \mu_{E_{2}}(1)     & \cdots & \mu_{E_{r}}(1)     \\
        \vdots             & \vdots             & \cdots & \vdots             \\
        \mu_{E_{1}}(n-r-1) & \mu_{E_{2}}(n-r-1) & \cdots & \mu_{E_{r}}(n-r-1)
\end{pmatrix}
\]
is an $(n-r) \times r$ matrix, and
\begin{align}
\bm{\mathcal{M}}_{22} = \begin{pmatrix}
\mu_{E_{1}}(n-r)   & \mu_{E_{2}}(n-r)   & \cdots & \mu_{E_{r}}(n-r)    \\
\mu_{E_{1}}(n-r+1) & \mu_{E_{2}}(n-r+1) & \cdots & \mu_{E_{r}}(n-r+1)\\
\vdots             & \vdots             & \cdots & \vdots             \\
\mu_{E_{1}}(n-1)   & \mu_{E_{2}}(n-1)   & \cdots & \mu_{E_{r}}(n-1)    
\end{pmatrix}
\end{align}
is an $r \times r$ matrix. Using \eqref{defn of mu}, $\bm{\mathcal{M}}_{22}$ is expressed by
\begin{align} \label{eg 1 M22}
\bm{\mathcal{M}}_{22} = \begin{pmatrix}
    a_{l_1} x_1^{l_1} & a_{l_1} x_2^{l_1} & \cdots & a_{l_1} x_r^{l_1}\\
    a_{l_2} x_1^{l_2} & a_{l_2} x_2^{l_2} & \cdots & a_{l_2} x_r^{l_2}\\
    \vdots            & \vdots            & \ddots & \vdots           \\
    a_{l_r} x_1^{l_r} & a_{l_r} x_2^{l_r} & \cdots & a_{l_r} x_r^{l_r}
\end{pmatrix}
\end{align}
where we recall $l_q = n-r+q-1$. One can further observe that
\begin{align} \label{eg 1 M22 decomp}
    \bm{\mathcal{M}}_{22} = D_a V D_x,
\end{align}
where $V$ is the Vandermonde matrix $V_{p,q} = x_{p}^{q-1}$ for $1 \leq p, q \leq r$, and 
\[
D_a = \operatorname{diag}(a_{l_1}, a_{l_2}, \cdots, a_{l_r}), \quad D_x = \operatorname{diag}(x_1^{n-r}, x_2^{n-r}, \cdots, x_r^{n-r}).
\]
Using the block structure of $\bm{\mathcal{M}}$, it can be readily shown that
\[
\overline{\bm{\mathcal{M}}}^{\top-1} = \begin{pmatrix}
\begin{array}{c|c}
    \bm I      & \bm{0} \\
    \noalign{\vskip 2pt\hrule\vskip 2pt}
    -\overline{\bm{\mathcal{M}}_{22}}^{\top-1} \overline{\bm{\mathcal{M}}_{12}}^{\top} & \overline{\bm{\mathcal{M}}_{22}}^{\top-1}
\end{array}
\end{pmatrix}
\]
where $\bm{0}$ is now the $(n-r) \times r$ zero matrix. 
Then, Theorem~\ref{thm:typeI} yields that 
\begin{align} \label{eg 1 expression 2}
\Psi_{\vec{n}}(z) = (b_0, b_1, \cdots, b_{n-r-1}, c_1, c_2, \cdots, c_r) \, \bm{\psi}(z) = \sum_{l = 0}^{n-r-1} b_l \, z^l +  \sum_{j=1}^r c_j \mathbf{1}_{E_j}(\arg z) 
\end{align}
where we recall \eqref{eg 1: varphi, psi} and the fact that $\bm{b} = (b_0, b_1, \cdots, b_{n-r-1})$ is the last row of $-\overline{\bm{\mathcal{M}}_{22}}^{\top-1} \overline{\bm{\mathcal{M}}_{12}}^{\top}$ and $\bm{c} = (c_1, c_2, \cdots, c_r)$ is the last row of $\overline{\bm{\mathcal{M}}_{22}}^{\top-1}$.
In other words, if we set $\bm{e}_r$ as the standard basis column vector $(0, \cdots, 0, 1)^\top \in \R^r$, we have
\[
\bm{e}_r^\top \, \overline{\bm{\mathcal{M}}_{22}}^{\top-1} = \bm{c} \iff \bm{\mathcal{M}}_{22} \, \overline{\bm{c}}^\top = \bm{e}_r.
\]
Using \eqref{eg 1 M22 decomp}, one can further simplify this relation as
\begin{equation}\label{eq:c-with-vand}
\overline{\bm{c}}^\top = a_{n-1}^{-1} \, D_x^{-1} V^{-1} \bm{e}_r.
\end{equation}
It is well-known, see for example \cite{Klinger}, that the inverse of the Vandermonde matrix is given by
\begin{equation} \label{eq:Vinv_entry}
(V^{-1})_{j,k} = \frac{(-1)^{r-k}\; e_{r-k} \left(t_1,\cdots,\widehat{t_j},\cdots,t_r \right)}
{\displaystyle\prod_{\substack{m=1 \\ m \neq j}}^r (t_j - t_m)},\quad 1\leq j,k \leq r,
\end{equation}
where $e_s(t_1,\cdots,\widehat{t_j},\cdots,t_r) \equiv e_s(\widehat{t_j})$ denotes the elementary symmetric polynomial of degree $s$ in the $r-1$ variables $\{t_\ell:\ell\neq j\}$.
This implies that the $j$-th entry of the last column $V^{-1} \bm{e}_r$ is given by
\begin{align} \label{last column of inv V}
(V^{-1} \bm{e}_r)_j = \prod_{\substack{m=1 \\ m \neq j}}^r \frac{1}{t_j - t_m}.
\end{align}

Therefore, from \eqref{eq:c-with-vand} we get 
\[
\left( \overline{\bm{c}}^\top \right)_j = a_{n-1}^{-1} \, x_j^{-(n-r)} \prod_{\substack{m=1 \\ m \neq j}}^r \frac{1}{x_j - x_m}, \quad j = 1, \cdots, r.
\]
By taking the complex conjugate, one can readily obtain \eqref{eg 1 defn of cj}.
Moreover, using $\bm{c}$, one can express $\bm{b}$ as
\[
\bm{b} = -\bm{c} \, \overline{\bm{\mathcal{M}}_{12}}^{\top},
\]
which means its $l$-th entry satisfies
\begin{align} \label{eg 1 expression 3}
b_l = -\sum_{k=1}^r c_k \, \mu_{E_k}(-l), \quad l = 0, 1, \cdots, n-r-1.
\end{align}
Combining \eqref{eg 1 expression 2}, \eqref{eg 1 expression 3}, and \eqref{defn of mu} finishes the proof of \eqref{MOPUC type I eg 1}.

%%%%%%%%%%%%%%%%%%%%%%%%%%%%%%%%%%%%%%%%%

\subsection{Proof of Corollary~\ref{cor: MOPUC kernel eg 1}}

To obtain the kernel reduction formula in the considered setting, we first need to express $\bm{\pi}(\xi)$ explicitly. Recall \eqref{eq:q-def} and \eqref{eg 1 K0}; we then have
\[
\bm{\pi}(\xi) = \bm{\psi}_2(\xi) - \sum_{j=0}^{n-1} \xi^j \int_{\TT} s^{-j} \, \bm{\psi}_2(s) \, \frac{|\dd s|}{2\pi}.
\]
By \eqref{eg 1: varphi, psi}, the $p$-th entry of $\bm{\pi}(\xi)$ is
\[
\left(\bm{\pi}(\xi) \right)_p = \mathbf{1}_{E_p}(\arg \xi) - \sum_{j=0}^{n-1} \xi^j \int_{\TT} s^{-j} \, \mathbf{1}_{E_p}(\arg s) \, \frac{|\dd s|}{2\pi} = \mathbf{1}_{E_p}(\arg \xi) - \sum_{j=0}^{n-1} \xi^j \mu_{E_p}(-j), \quad p = 1, \cdots, r.
\]
By \eqref{eg 1 M22 decomp} and \eqref{eg 1: varphi, psi}, we have
\begin{align} \label{eg 1 M22 varphi2}
\begin{aligned}
\left( \bm{\mathcal{M}}_{22}^{-1} \,
\bm{\varphi}_2(z) \right)_p &= \sum_{k=1}^r x_p^{-(n-r)} (V^{-1})_{p,k} \, a_{l_k}^{-1} \, z^{l_k}\\
&= \frac{x_p^{-(n-r)}}{\prod_{\substack{m=1 \\ m \neq p}}^r (x_p - x_m)} \sum_{k=1}^r (-1)^{r-k} e_{r-k}(\hat{x}_p) a_{l_k}^{-1} \, z^{l_k},
\end{aligned}
\end{align}
where we used \eqref{eq:Vinv_entry} for the second equality.
Thus, by Theorem~\ref{thm:kernel}, one can readily obtain \eqref{MOPUC kernel eg 1}.

%%%%%%%%%%%%%%%%%%%%%%%%%%%%%%%%%%%%%%%%%

\subsection{Proof of Corollary~\ref{cor: MOPUC type II eg 2}}

Assume that
\begin{equation} \label{eg 2 assumption}
x_j=x_1\rho^{j-1}, \qquad \rho=e^{2\pi i/r}, \qquad j=1,\ldots,r.
\end{equation}
Then \(x_1,\ldots,x_r\) are precisely the roots of
\begin{equation} \label{eg 2 defn of W}
W(x):=x^r-x_1^r.
\end{equation}
Hence
\begin{equation} \label{eg 2 W vieta}
\prod_{j=1}^r (x-x_j)=W(x)=x^r-x_1^r.
\end{equation}
By Vieta's formulas,
\begin{equation} \label{eg 2 claim on esp}
e_k(x_1,\ldots,x_r)=\begin{dcases} 0, & 1\leq k\leq r-1,\\ (-1)^{r-1}x_1^r, & k=r. \end{dcases}
\end{equation}
Substituting \eqref{eg 2 claim on esp} into \eqref{MOPUC type II eg 1}, all terms vanish except the term corresponding to \(q=1\). Since \(l_1=n-r\), we obtain
\[
\Phi_{\vec n}(z)=z^n-\left(\frac{a_n}{a_{n-r}}x_1^r\right)z^{n-r}=z^{n-r}\left(z^r-\frac{a_n}{a_{n-r}}x_1^r\right).
\]
Thus \(z=0\) is a zero of multiplicity \(n-r\). The nonzero zeros satisfy \(z^r=(a_n/a_{n-r})x_1^r\). Writing \(x_1=e^{i\theta_1}\) and \(\frac{a_n}{a_{n-r}}=\left|\frac{a_n}{a_{n-r}}\right|e^{i\arg(a_n/a_{n-r})}\), with \(\arg(a_n/a_{n-r})\in\{0,\pi\}\), gives
\[
\zeta_k=\left|\frac{a_n}{a_{n-r}}\right|^{1/r}\exp\left(i\left(\theta_1+\frac{\arg(a_n/a_{n-r})+2\pi k}{r}\right)\right), \qquad k=0,\ldots,r-1.
\]
Finally, by the definition of \(a_\ell\),
\[
\left|\frac{a_n}{a_{n-r}}\right|=\left|\frac{n-r}{n}\frac{\sin\left(\frac{n\Delta}{2}\right)}{\sin\left(\frac{(n-r)\Delta}{2}\right)}\right|,
\]
which gives the stated formula. Since \(a_n a_{n-r}\neq 0\), these \(r\) nonzero zeros are simple.
\subsection{Proof of Corollary~\ref{cor: MOPUC type I eg 2}}

Under the equally spaced assumption, one can simplify \(c_j\) in \eqref{eg 1 defn of cj}. By \eqref{eg 2 W vieta}, we have
\[
W'(x_j)=\prod_{\substack{m=1\\ m\neq j}}^r (x_j-x_m).
\]
On the other hand, \eqref{eg 2 defn of W} gives \(W'(x_j)=r x_j^{r-1}\). Therefore
\[
c_j=\frac{(-1)^{r-1}x_j^{n-2}}{a_{n-1}}\frac{\prod_{m=1}^r x_m}{r x_j^{r-1}}.
\]
Since \(\prod_{m=1}^r x_m=e_r(x_1,\ldots,x_r)=(-1)^{r-1}x_1^r\) by \eqref{eg 2 claim on esp}, we obtain
\[
c_j=\frac{x_j^{n-2}x_1^r}{r a_{n-1}x_j^{r-1}}.
\]
Finally, since \(x_j^r=x_1^r\), this becomes
\[
c_j=\frac{x_j^{n-1}}{r a_{n-1}}.
\]
Substituting this expression for \(c_j\) into \eqref{MOPUC type I eg 1} gives \eqref{MOPUC type I eg 2}.

%%%%%%%%%%%%%%%%%%%%%%%%%%%%%%%%%%%%%%%%%

\subsection{Proof of Corollary~\ref{cor: MOPUC kernel eg 2}}

For \(p=1,\ldots,r\), define
\begin{equation} \label{eg 2 defn of S}
S(x)=\sum_{i=0}^{r-1}x_p^{r-1-i}x^i.
\end{equation}
Then \(S(x)=(x^r-x_p^r)/(x-x_p)\). Since \(x_p^r=x_1^r\) by \eqref{eg 2 assumption}, it follows from \eqref{eg 2 defn of W} and \eqref{eg 2 W vieta} that
\begin{equation} \label{eg 2 expression of S}
S(x)=\frac{W(x)}{x-x_p}=\prod_{\substack{m=1\\ m\neq p}}^r (x-x_m)=\sum_{k=0}^{r-1}(-1)^{r-1-k}e_{r-1-k}(\hat{x}_p)x^k.
\end{equation}
Comparing \eqref{eg 2 defn of S} and \eqref{eg 2 expression of S}, we obtain
\[
(-1)^{r-k}e_{r-k}(\hat{x}_p)=x_p^{r-k}, \qquad k=1,\ldots,r.
\]
Using this relation in \eqref{eg 1 M22 varphi2}, and since \(W'(x_p)=r x_p^{r-1}\), one has
\[
\left( \bm{\mathcal{M}}_{22}^{-1} \bm{\varphi}_2(z) \right)_p = \frac{x_p^{-(n-r)}}{W'(x_p)}\sum_{k=1}^r x_p^{r-k}a_{l_k}^{-1}z^{l_k}=\frac{1}{r}\sum_{k=1}^r\frac{z^{l_k}}{a_{l_k}x_p^{l_k}}=\frac{1}{r}\sum_{k=1}^r\frac{z^{l_k}}{a_{l_k}x_1^{l_k}}\rho^{-(p-1)l_k}.
\]
Therefore, the kernel reduction formula under the equally spaced assumption becomes
\[
\mathcal K(z,\xi)-\mathcal K_0(z,\xi)=\frac{1}{r}\sum_{p=1}^r\left(\mathbf{1}_{E_p}(\arg \xi)-\sum_{j=0}^{n-1}\xi^{-j}a_jx_p^j\right)\left(\sum_{k=1}^r\frac{z^{l_k}}{a_{l_k}x_1^{l_k}}\rho^{-(p-1)l_k}\right),
\]
which gives \eqref{MOPUC kernel eg 2} after changing the order of summation.

\section{External Source Models}\label{section baik}

In this section, we explain how the kernel reduction formula in Theorem~\ref{kernel reduction}
recovers two earlier formulae for random matrix models with external source.
First, we recover Baik's finite-rank reduction formula \cite[Theorem 1]{Bai09}.  We then explain how
the same mixed-moment formalism also contains the earlier formula
of Zinn-Justin \cite[Equation (3.10)]{ZinnJustin1997} for the kernel in terms of external source parameters and the orthogonal polynomial with respect to the reference weight.

\subsection{A review of the external source model}

The connection between multiple orthogonal polynomials and random matrix theory
became especially prominent through Hermitian random matrix models with external
source.  In such models, introduced by Brezin and Hikami \cite{BreHi1, BreHi2}, one considers probability measures of the form
\begin{align} \label{prob distribution}
    \frac{1}{Z_n}\exp\{-\operatorname{Tr}(V(M)-AM)\}\,\dd M
\end{align}
on Hermitian matrices, where $A$ is a fixed Hermitian matrix.  If $A$ has
distinct eigenvalues $a_1,\cdots,a_r$ with multiplicities
$n_1,\cdots,n_r$, then the average characteristic polynomial
\[
        \mathbb E[\det(zI-M)]
\]
is characterized by multiple orthogonality with respect to the weights
\[
        e^{-(V(x)-a_jx)}, \qquad j=1,\cdots,r.
\]
This observation, due to Bleher and Kuijlaars \cite{BK0}, showed that the
external-source model is governed by multiple orthogonal polynomials. Moreover,
the eigenvalue correlation functions form a determinantal point process whose
kernel can be written in terms of type I and type II multiple orthogonal
polynomials.  The Christoffel-Darboux formula of Daems and Kuijlaars
\cite{DaemsKuijlaars2004} then provided a compact expression for this kernel, generalizing the
ordinary Christoffel-Darboux formula for orthogonal polynomial ensembles.

The Gaussian external-source model with two eigenvalues $\pm a$ of equal
multiplicity has played a central role in this development.  In the
subcritical regime $0<a<1$, analyzed by Aptekarev, Bleher, and Kuijlaars,
the limiting spectrum is supported on a single interval, as in the case with
no external source \cite{ABK05}.  As $a$ increases, the external field pushes the spectrum
toward a splitting transition; in the supercritical regime $a>1$, analyzed by
Bleher and Kuijlaars, the support consists of two intervals
\cite{BK04b}.  In both non-critical regimes, the local correlations have
the standard universal limits: the sine kernel in the bulk and the Airy kernel
near regular edges.  At the critical value $a=1$, however, the two regimes
meet at the point where one band is about to split into two, and the local
correlations are governed instead by the Pearcey kernel \cite{BK07}.  For more general even polynomial
potentials, Bleher, Delvaux, and Kuijlaars later showed that the limiting
eigenvalue distribution is described by a constrained vector equilibrium
problem \cite{BDK11}.  The analysis of these models was based on a $3\times 3$ Riemann-Hilbert
problem as established in \cite{VanAsscheGeronimoKuijlaars2001}.  In this sense, these works provided some of the first large-$n$
steepest descent analyses beyond the classical $2\times 2$ setting, and
helped inform subsequent extensions of the Deift-Zhou nonlinear steepest
descent method to higher-dimensional Riemann-Hilbert problems
\cite{DaemsKuijlaarsVeys2008,
DuitsKuijlaars2009,KuijlaarsMartinezFinkelshteinWielonsky2009,
BertolaGekhtmanSzmigielski2013}.

A different, but closely related, line of work concerns finite-rank, or
spiked, external sources.  In this setting, the external source $A$ has only
finitely many nonzero eigenvalues, or more generally, a number of nonzero
eigenvalues that grows slowly with $n$.  Bertola--Buckingham--Lee--Pierce
considered the small-rank external source
\[
    A=\operatorname{diag}(\underbrace{a,\cdots,a}_{r},\underbrace{0,\cdots,0}_{n-r}),
\]
with $r$ fixed or slowly growing with $n$.  For general analytic potentials,
they proved universality in the subcritical and supercritical regimes, and, in
the critical and near-critical regimes with $r=\mathcal O(n^\gamma)$, obtained
the corresponding $r$-Airy edge statistics \cite{BBLP12,BBLP13}.

Another important perspective was developed by Baik, who showed in
\cite{Bai09} that for
\begin{align} \label{external source A}
    A=\operatorname{diag}(\underbrace{a_1,\cdots,a_r}_{r},\underbrace{0,\cdots,0}_{n-r}),
\end{align}
the multiple-orthogonal Christoffel-Darboux kernel can be expressed as a finite-rank perturbation of the ordinary Christoffel-Darboux kernel associated with the reference weight $e^{-V(x)}$.
Baik and Wang used this perspective in their analysis of the largest eigenvalue of Hermitian matrix models with spiked external source, first in the rank-one case and then in higher rank \cite{BaikWang1, BaikWang2}.
Their results exhibit the familiar subcritical, critical, and supercritical regimes: below the critical value, the largest eigenvalue remains at the edge of the referenced equilibrium measure, while above it one or more outliers detach from the bulk.

%\magenta{Do some work and write this  better:}\blue{This point of view also has natural connections with other parts of random matrix theory.  Multiple orthogonal polynomials underlie multiple orthogonal polynomial ensembles, which form a distinguished class of determinantal point processes \cite{Kui10}.  They also appear in models of non-intersecting paths, where the Gaussian external-source model has an equivalent Brownian motion interpretation.  Beyond external-source ensembles, closely related biorthogonal and polynomial ensembles arise in the study of products and sums of random matrices, characteristic polynomial averages, and ratios of characteristic polynomials; see, for example, \cite{BDS03,ASW20}.  In this sense, the reduction formulae proved here may be viewed as finite-dimensional analogues of the transformations that often simplify kernels, characteristic polynomial averages, and gap probabilities in determinantal random matrix models.}

\subsection{Recovery of Baik's formula}

\begin{comment}
\textcolor{purple}{[Remove from here until}
In this section, we recover the kernel reduction formula by Baik \cite[Theorem 1]{Bai09} from Theorem~\ref{kernel reduction}. First, we will briefly explain \textcolor{red}{his setting}.
Let $M \in \mathcal{H}_n$ be a random Hermitian matrix with respect to the probability distribution
\begin{align}
    \dd \mu_n(M) = \frac{1}{Z_n} \exp\{-\operatorname{Tr}(V(M) - AM)\} \dd M, \label{dmun external source}
\end{align}
where $-V(x)$ is a sufficiently fast decaying function on $\R$, $A = \operatorname{diag}(a_1, \cdots, a_r, 0, \cdots, 0)$ such that $a_1 > \dots > a_r$ and $a_j \neq 0$ for all $j = 1, \cdots, r$, and the normalization constant is given by
\begin{align}
    Z_n = \int_{\mathcal{H}_n} \exp\{-\operatorname{Tr}(V(M) - AM)\} \dd M.
\end{align}
Let $P_{\vec{n}}(z)$ be the average characteristic polynomial defined by
\begin{align}
    P_{\vec{n}}(x) = \mathbb{E} \left[ \operatorname{det}(xI - M) \right]
    \label{ave char poly external source}
\end{align}
where the expectation is with respect to \eqref{dmun external source}. This turns out to be a type II MOPRL \cite{BK0}. 
\textcolor{purple}{here.]}
\end{comment}

%\textcolor{purple}{
First, we describe the setting considered by Baik in \cite{Bai09} more precisely. Let $M$ be a random Hermitian matrix with respect to the probability distribution \eqref{prob distribution} where $A$ is given as in \eqref{external source A} such that $a_1 > \dots > a_r$ and $a_j \neq 0$ for all $j = 1, \cdots, r$.
%}
Then, the average characteristic polynomial is a type II MOPRL defined by the multi-index
\begin{align} \label{Baik vector}
\vec{n} = (n_1, n_2, \cdots, n_{r+1}) = (n-r, 1, \cdots, 1) \in \N^{r+1}
\end{align}
and the orthogonality condition \eqref{type II orthogonality} with 
\begin{align} \label{Baik weights}
w^{(1)}(x) = e^{-V(x)} \quad \text{and} \quad w^{(j+1)}(x) = e^{-(V(x) - a_j x)}, \quad j=1,\cdots,r.
\end{align} 
%\textcolor{purple}{
The main object of study in \cite{Bai09} is the Christoffel-Darboux kernel $K^{B}$ associated with this MOPRL, together with a reduction formula expressing it in terms of the Christoffel-Darboux kernel $K^{B}_{0}$ corresponding to the monic orthogonal polynomial $p_n(x)$ with respect to the weight $w^{(1)}(x)=e^{-V(x)}$.
%}
%\red{His main concern is} about the Christoffel-Darboux kernel $K^{B}$ of this MOPRL and its reduction formula to the Christoffel-Darboux kernel $K^B_{0}$ corresponding to the monic orthogonal polynomial $p_n(x)$ with respect to the weight $w^{(1)}(x) = e^{-V(x)}$.
(In \cite{Bai09}, the kernels, $K^B$ and $K^B_{0}$ are called $\mathcal K$ and $\mathcal K_0$ respectively, but to avoid confusion with our notation for the kernels for MOPUC, we denote those by Baik in block letter with superscript $B$).

\begin{lemma}
For the orthogonal weights \eqref{Baik weights}, %\textcolor{purple}{
our kernels are a scaled version of Baik's,
%}
\begin{equation} \label{kernel comparison}
  K(x,y) = K^B(x,y) \dfrac{w^{(1)}(y)}{w^{(1)}(x)}, \quad K_0(x,y) = K^B_{0}(x,y) \dfrac{w^{(1)}(y)}{w^{(1)}(x)}.
\end{equation}
\end{lemma}

\begin{proof}
Our kernel $K(x,y)$ defined in \eqref{defn of CD kernel for p} is exactly the same as $K_n(x,y)$ defined in \cite[Equation (1.8)]{DaemsKuijlaars2004} by Daems-Kuijlaars. The authors characterized this kernel through the solution $Y(z)$ of the $(m+1) \times (m+1)$-Riemann-Hilbert problem, see \cite[Section 2]{DaemsKuijlaars2004}. The key result is equation (2.10) in \cite{DaemsKuijlaars2004}, which we copied below,
\begin{align} \label{DK equation 2.10}
    (x - y) K_n(x, y) = \frac{1}{2\pi i} \begin{pmatrix}
        0 & w_1(y) & \cdots & w_m(y)
    \end{pmatrix} Y^{-1}(y) Y(x) \begin{pmatrix}
        1\\
        0\\
        \vdots\\
        0
    \end{pmatrix}.
\end{align}
Since \cite{DaemsKuijlaars2004} considered a multi-index $\vec{n} = (n_1, n_2, \cdots, n_m)$ and the weights $w_1(y), w_2(y), \cdots, w_m(y)$ in \eqref{DK equation 2.10}, comparing them with \eqref{Baik vector} and \eqref{Baik weights}, equation \eqref{DK equation 2.10} gives the characterization of our kernel $K(x, y)$ with Daems-Kuijlaars' $(r+2) \times (r+2)$-matrix valued function $Y(z)$,
\begin{align} \label{our kernel wrt RH solution}
    (x - y) K(x, y) = \frac{1}{2\pi i} \begin{pmatrix}
        0 & e^{-V(y)} & e^{-(V(y)-a_1x)} & \cdots & e^{-(V(y)-a_rx)}
    \end{pmatrix} Y^{-1}(y) Y(x) \begin{pmatrix}
        1\\
        0\\
        \vdots\\
        0
    \end{pmatrix}.
\end{align}
On the other hand, Baik defined his kernel $K^B(x,y)$ in \cite[Equation (8)]{Bai09} by
\begin{align} \label{Baik equation 8}
    K^B(x, y) = \frac{e^{-V(x)}}{2 \pi i (x-y)} \begin{pmatrix}
        0 & 1 & e^{a_1 y} & \cdots & e^{a_r y}
    \end{pmatrix} \mathbf{Z}_{\pm}^{-1}(y) \mathbf{Z}_{\pm}(x) \begin{pmatrix}
        1\\
        0\\
        \vdots\\
        0
    \end{pmatrix},
\end{align}
where $\mathbf{Z}(z)$ is a solution of the $(r+2)\times(r+2)$-Riemann-Hilbert problem written in \cite[pages 2-3]{Bai09}. Comparing this $\mathbf{Z}$-RHP with the $Y$-RHP of Daems-Kuijlaars, one can readily see that they are identical. By \eqref{our kernel wrt RH solution} and \eqref{Baik equation 8}, we obtain the first equality in \eqref{kernel comparison}.

The other kernel comparison is rather elementary. Baik defined $K^B_0(x,y)$ by \cite[Equation (14)]{Bai09}, which we copied below,
\[
K^B_0(x,y) = \frac{e^{-V(x)} \kappa_{n-1}^2 \left( p_{n}(x) p_{n-1}(y) - p_{n-1}(x) p_{n}(y) \right)}{x-y},
\]
where $p_{k}(z)$ is the monic orthogonal polynomial of degree $k$ with respect to the weight $e^{-V(z)}$ and $\kappa_k$ is the positive constant such that $\{ \kappa_k p_{k}(z) \}_{k\geq 0}$ becomes a sequence of orthonormal polynomials. 
Recall that our kernel $K_0(x,y)$ is defined by \eqref{defn of CD kernel for p} and standard Christoffel-Darboux formula
\[
\sum_{j=0}^{n-1}\frac{p_j(x)p_j(y)}{h_j} = \frac{1}{h_{n-1}}\frac{p_n(x)p_{n-1}(y) - p_{n-1}(x)p_n(y)}{x-y}.
\]
It is a well-known fact that the norm $h_k$ and the constant $\kappa_k$ are related by $h_k = 1/\kappa_k^2$. This all gives the second equality in \eqref{kernel comparison}.
\end{proof}

Baik's main result \cite{Bai09} is the following.
\begin{theorem}[Baik]
It holds that
\begin{equation}\label{Baik's formula}
    \left( K^B(x,y)-K^B_0(x,y) \right) e^{V(x)}= \mathbf{w}(y)^\top B^{-1}\mathbf{t}(x)
\end{equation}
where 
\begin{itemize}
    \item $\mathbf{t}(x) = \left( p_{n-r}(x), p_{n-r+1}(x), \cdots, p_{n-1}(x) \right)^\top$ is an $r$-column vector,
    \item $\mathbf{v}(y) = (e^{a_1 y}, e^{a_2 y}, \cdots, e^{a_r y})^\top$ is an $r$-column vector,
    \item $\mathbf{w}(y)^\top = \mathbf{v}(y)^\top - \int_\R K^B_{0}(s,y)\mathbf{v}(s)^\top \dd s$ is an $r$-row vector, and
    \item $B$ is an $r \times r$ moment matrix defined by $B = \int_\R \mathbf{t}(s)\mathbf{v}(s)^\top e^{-V(s)} \dd s$.
\end{itemize}
\end{theorem}

It is immediate to see that Baik's vector $\mathbf{t}(x)$ is exactly our $\bm{p}_2(x)$, recall \eqref{defn of pi} and \eqref{defn of pi phi v, 1-2}. Also, one can readily observe that our $\bm{q}_2(y)$ can be expressed by
\begin{align} \label{our phi2 is baiks v}
    \bm{q}_2(y) = \mathbf{v}(y) w^{(1)}(y)
\end{align}
using Baik's vector $\mathbf{v}(y)$ by recalling \eqref{defn of phi} and \eqref{defn of pi phi v, 1-2}.
Then, our $\bm{M}_{22}$ becomes exactly Baik's matrix $B$,
\[
\bm{M}_{22} = \int_\R \bm{p}_2(x) \bm{q}_2(x)^\top \dd x = \int_\R \mathbf{t}(x) \mathbf{v}(x)^\top w^{(1)}(x) \dd x = B,
\]
where we used \eqref{defn of M} and \eqref{defn of M, H 4 blocks}.
Substituting \eqref{our phi2 is baiks v} and \eqref{kernel comparison} into our $\bm{r}(y)$, see \eqref{defn of q}, to have
\begin{align*}
    \bm{r}(y) %&= \bm{q}_2(y) - \int_\R K_0(s,y) \bm{q}_2(s) \, \dd s \\
    &= \mathbf{v}(y)w^{(1)}(y) - \int_\R K^B_0 (s,y) \frac{w^{(1)}(y)}{w^{(1)}(s)} \mathbf{v}(s) w^{(1)}(s) \, \dd s \\
    &= w^{(1)}(y) \left( \mathbf{v}(y) - \int_\R K_{0}^B(s,y) \mathbf{v}(s) \, \dd s \right).
\end{align*}
The term in the parenthesis is exactly Baik's vector $\mathbf{w}(y)$, leading to the simple relation,
\begin{equation}
    \bm{r}(y) = w^{(1)}(y) \mathbf{w}(y).
\end{equation}
Therefore, Theorem~\ref{kernel reduction} becomes
\begin{align}
K(x,y)-K_0(x,y) &= \bm{r}(y)^\top \bm{M}_{22}^{-1} \bm{p}_2(x) = (w^{(1)}(y) \mathbf{w}(y))^\top B^{-1} \mathbf{t}(x) = w^{(1)}(y) \mathbf{w}(y)^\top B^{-1} \mathbf{t}(x).
\end{align}
Using \eqref{kernel comparison}, we recovered Baik's formula,
\begin{equation}
\left( K^B(x,y)-K^B_{0}(x,y) \right) e^{V(x)}= \mathbf{w}(y)^{\top} B^{-1}\mathbf{t}(x).
\end{equation}

\subsection{Recovery of Zinn-Justin's formula}
The same mixed-moment formalism also contains the finite-$N$ external-source
kernel of Zinn-Justin \cite{ZinnJustin1997}, up to a scalar which does not affect the correlation determinant.
Since this work appeared before the aforementioned work
of Bleher and Kuijlaars \cite{BK0}, the kernel was not expressed in terms of
multiple orthogonal polynomials. 

Let us introduce this kernel and show its correlation determinant agrees with that of $K(x,y)$ given by \eqref{defn of CD kernel}.   Write
\[ W_N(x):=e^{-N V(x)} \]
and consider an external source
\[  A=\operatorname{diag}(a_0,a_1,\cdots,a_{N-1}). \]
Let $p_k$ be the monic orthogonal polynomials associated with
$W_N(x)$, and define
\begin{equation}\label{ZJ FG}
    F_k(x)=h_k^{-1/2}p_k(x)W_N(x)^{1/2},
    \quad
    G_\ell(x)=e^{N a_\ell x}W_N(x)^{1/2},
\end{equation}
where
\[h_k=\int_{\mathbb R}p_k(x)^2W_N(x)\,\dd x.\]
Let $\boldsymbol\alpha$ denote Zinn-Justin's inverse moment matrix, namely
\begin{equation}\label{ZJ al}
    \boldsymbol\alpha
    =
    \boldsymbol E^{-1},
    \quad
    E_{\ell k}
    =
    \int_{\mathbb R}G_\ell(x)F_k(x)\,\dd x,
    \quad 0\leq \ell,k\leq N-1 .
\end{equation}
Zinn-Justin's kernel is then expressed in terms of the external source data
and the ordinary orthogonal polynomials as
\begin{equation}\label{ZJ Kernel}
    K_{\mathrm{ZJ}}(x,y)
    =
    \sum_{k,\ell=0}^{N-1}
        F_k(x)\alpha_{k\ell}G_\ell(y).
\end{equation}

We now compare this kernel with ours.  After replacing $V(x)$ by
$V(x)-a_0x$ and replacing $a_j$ by $a_j-a_0$, we may assume $a_0=0$.
If we take \begin{equation}\label{ZJ weights}
    \vec n=(1,1,\cdots,1)\in\mathbb N^N,  \quad
    w^{(j+1)}(x)=e^{N a_jx}W_N(x),
    \quad j=0,\cdots,N-1,
\end{equation}
then, recalling \eqref{defn of phi}, the vector $\boldsymbol{q}(x)$ in our
mixed-moment construction is
\[
    \boldsymbol{q}(x)=
    \bigl(
        W_N(x),
        e^{N a_1x}W_N(x),
        \cdots,
        e^{N a_{N-1}x}W_N(x)
    \bigr)^{\top}.
\]
Thus
\[  M_{k\ell}
    =
    \int_{\mathbb R}p_k(s)e^{N a_\ell s}W_N(s)\,\dd s,
    \quad k,\ell=0,\cdots,N-1.
\]
  We have $ E_{\ell k} = h_k^{-1/2}M_{k\ell}$ and  Therefore, 
\[
    \boldsymbol{E}
    =\boldsymbol{M}^{\top}\boldsymbol{H}^{-1/2},
\mbox{ \ and thus \ }
    \boldsymbol\alpha
    =
\boldsymbol{H}^{1/2}\boldsymbol{M}^{-\top}.
\]
Substituting this relation into Zinn-Justin's formula \eqref{ZJ Kernel}, we get
\begin{equation}
    \label{KZ}
    K_{\mathrm{ZJ}}(x,y)
=W_N(x)^{1/2}W_N(y)^{1/2}
    \sum_{k,\ell=0}^{N-1}
        p_k(x)
        \left(\boldsymbol{M}^{-\top}\right)_{k\ell}
        e^{N a_\ell y}.
\end{equation}

Now let us focus on the kernel $K(x,y)$ \eqref{defn of CD kernel} corresponding to the multi-index and weights given by \eqref{ZJ weights}. 
%\sout{Before the block reduction used in Theorem \ref{kernel reduction}, the same linear algebraic steps give the (full matrix) kernel}
%\textcolor{purple}{
Recall Lemma~\ref{lem:kernel uniqueness} where we obtain
%}
\[
    K(x,y)=\boldsymbol{q}(y)^{\top}\boldsymbol{M}^{-1}\boldsymbol{p}(x).
\]
Since
\[
    \boldsymbol{q}(y)^{\top}
    =
    W_N(y)
    \bigl(
        1,e^{N a_1y},\cdots,e^{N a_{N-1}y}
    \bigr),
\]
the kernel can be written as
\[
    K(x,y)
    =
    W_N(y)
    \sum_{k,\ell=0}^{N-1}
        p_k(x)
        (\boldsymbol{M}^{-\top})_{k\ell}
        e^{N a_\ell y}.
\]
Comparing this with \eqref{KZ} gives
\[
    K_{\mathrm{ZJ}}(x,y)
    =
    \frac{W_N(x)^{1/2}}{W_N(y)^{1/2}}K(x,y).
\]
Consequently, for every $m\leq N$,
\begin{equation}\label{kernel dets agree}
    \det\bigl[K_{\mathrm{ZJ}}(x_i,x_j)\bigr]_{i,j=1}^{m}
    =
    \det\bigl[K(x_i,x_j)\bigr]_{i,j=1}^{m},
\end{equation}
because the scalar factors cancel between rows and columns.

\begin{remark}
    Baik's kernel reduction can also be viewed as a direct consequence of
Zinn-Justin's formula \eqref{ZJ Kernel}. Indeed, after taking the
confluent limit $a_{r+1}, a_{r+2}, \cdots, a_n \to 0$\footnote{Notice, one can not simply put $a_{r+1}= a_{r+2}= \cdots= a_n = 0$ as this would make the matrix $\left(
        \int_{\mathbb R}G_\ell(x)F_k(x)\,dx
    \right)$ singular, see \eqref{ZJ FG} and \eqref{ZJ al}.}, and similar linear algebraic steps used above one can precisely obtain the
$r\times r$ correction in Baik's formula \eqref{Baik's formula}.
\end{remark}

\appendix

\section{Relation between the Mixed-moment Matrix and Various Hankel Matrices}
\label{Mixed-moment and hankel}

This appendix is slightly off the mainstream of the paper, but our linear algebraic technique on the mixed-moment matrix $\bm{M}$ connects some known results, so we present them here.

\subsection{Alternative Proof of \cite[Proposition 1.9]{GL}}
\label{sec:mixed-vs-bordered}

In this subsection, we fix a vector,
\begin{align} \label{app bordered hankel vector}
\vec{n} = (n-m, 1, 1, \cdots, 1) \in \R^n,
\end{align}
where $n > m \ge 1$, and a vector of measures,
\begin{align} \label{app bordered hankel measures}
\vec{\mu}(x) = \left( \mu(x), \delta(x - z_1), \delta(x - z_2),  \cdots, \delta(x - z_m) \right),    
\end{align}
where $\delta(x - z_j)$, $z_j \in \R$, is the dirac delta and $\mu(x)$ is a measure on $\mathscr I \subseteq \mathbb R$ with finite moments $\mu_k := \int_{\mathscr I} x^k\,\dd\mu(x)$. We denote by $\{p_k(x)\}_{k\ge0}$ the \textit{monic} orthogonal polynomials with respect to $\mu$ satisfying
\begin{equation}\label{eq:orthog}
\int_{\mathscr I} p_i(x)\, p_j(x)\,\dd\mu(x)=h_i\,\delta_{ij},\quad h_i>0.
\end{equation}
In this setting, we shall give a simpler proof of \cite[Proposition 1.9]{GL} using solely linear algebra. Here is the statement.

\begin{proposition} \label{app bordered hankel prop}
The determinant of the bordered Hankel matrix
\begin{align} \label{borderd hankel defn}
\bm H_n^{B}[\mu;z_1,\cdots,z_m] = \begin{pmatrix}
		\mu_0 & \mu_1 & \cdots & \mu_{n-m-1} & 1 & 1 & \cdots & 1 \\
		\mu_1 & \mu_2 & \cdots & \mu_{n-m} & z_1 & z_2 & \cdots & z_m \\
		\mu_2 & \mu_3 & \cdots & \mu_{n-m+1} & z_1^2 & z_2^2 & \cdots & z_m^2 \\
		\vdots & \vdots & \reflectbox{$\ddots$} & \vdots & \vdots & \vdots & \ddots & \vdots  \\
		\mu_{n-1} & \mu_n & \cdots & \mu_{2n-m-2} & z_1^{n-1} & z_2^{n-1} & \cdots & z_{m}^{n-1}\end{pmatrix}
\end{align}
satisfies
\[
\det\bm H_n^{B}[\mu;z_1,\cdots,z_m] = \det \bm{H}_{n-m}[\mu] \det 
\begin{pmatrix}
	p_{n-m}(z_1) & \cdots & p_{n-m}(z_{m}) \\
	\vdots & \ddots & \vdots \\
	p_{n-1}(z_1) & \cdots & p_{n-1}(z_{m})
\end{pmatrix}
\]
where \[\det \bm{H}_{n-m}[\mu] = \prod_{k=0}^{n-m-1} h_k.\]
\end{proposition}

\begin{proof}
We define two $n$-column vectors,
\[
\bm p_{n}(x) = \begin{pmatrix}
    \bm p_{1}(x)\\
    \bm p_{2}(x)
\end{pmatrix}, \quad \bm m_{n}(x) = \begin{pmatrix}
    \bm m_{1}(x)\\
    \bm m_{2}(x)
\end{pmatrix},
\]
where
\begin{align}
&\bm p_{1}(x) = \left(p_0(x), p_1(x), \cdots,p_{n-m-1}(x) \right)^\top, \quad \bm p_{2}(x) = \left( p_{n-m}(x), p_{n-m+1}(x), \cdots, p_{n-1}(x) \right)^\top,\\
&\bm m_{1}(x) = \left( 1,x,\cdots,x^{n-m-1} \right)^\top, \quad \bm m_{2}(x) = \left( x^{n-m}, x^{n-m+1}, \cdots, x^{n-1} \right)^\top.
\end{align}
As the set of monomials $\{x^k\}_{k=0}^{n-1}$ and the set of monic orthogonal polynomials $\{ p_k(x)\}_{k=0}^{n-1}$ form bases of the space of degree $n-1$ polynomials, there exist a change of bases matrix $L_{n}$ such that it is an $n \times n$ lower-triangular matrix with all diagonal entries $1$ and satisfies \( \bm{p}_n(x) = L_n \bm m_n(x) \).
We denote the $(n-m) \times (n-m)$ upper left block of $L_{n}$ by $L_{n}^{[n-m]}$ that satisfies \( \bm{p}_{1}(x) = L_n^{[n-m]} \bm m_{1}(x) \).

Next, we define the $n \times (n-m)$ moment matrix with respect to $\mu$,
\[
G^{(1)}[\mu] = \int_{\mathscr I} \bm m_n (x)\bm m_{1} (x)^\top\,\dd\mu(x)=[\mu_{i+k}]_{\substack{0\le i\le n-1\\ 0\le k\le n-m-1}}.
\]
Then, one can readily observe that 
\begin{equation} \label{app key eqn}
\int_{\mathscr I} \bm p_n (x) \bm p_{1} (x)^\top \,\dd\mu = L_n G^{(1)}[\mu] \left( L_n^{[n-m]} \right)^\top.
\end{equation}
However, orthogonality \eqref{eq:orthog} simply implies that
\[
\int_{\mathscr I} \bm p_n (x) \bm p_{1} (x)^\top \,\dd\mu = \begin{pmatrix}
    \bm{H}_{n-m}[\mu]\\
    \bm{0}
\end{pmatrix} \in \mathrm{Mat}_{n \times (n-m)} \R,
\]
where \(\bm{H}_{n-m}[\mu] = \mathrm{diag}(h_0, h_1, \cdots, h_{n-m-1})\) (in the main text, this was called $\bm H_1$, see \eqref{defn of H} and \eqref{defn of M, H 4 blocks}).
Thus, the mixed-moment matrix $\bm{M}$ for a vector \eqref{app bordered hankel vector} and measures \eqref{app bordered hankel measures} is given by
\begin{align} \label{app bordered hankel M}
\bm{M} &= \left(\int_{\mathscr I} \bm p_n(x) \bm p_{1}(x)^\top \dd\mu(x), \bm p_n(z_1), \cdots, \bm p_n (z_m) \right) = \begin{pmatrix}
\bm H_{n-m}[\mu] & \bm M_{12} \\[2pt]
\bm 0    & \bm M_{22}
\end{pmatrix},
\end{align}
where $\bm M_{12} = \left(\bm p_1(z_1)\ \cdots\ \bm p_1(z_m)\right)$ is an $(n-m) \times m$ matrix and $\bm M_{22} = \left(\bm p_2(z_1) \cdots \bm p_2(z_m)\right)$ is an $m \times m$ matrix.

Since the bordered Hankel matrix \eqref{borderd hankel defn} can be written as
\[
\bm H_n^{B}[\mu;z_1,\cdots,z_m]
= \begin{pmatrix} G^{(1)}[\mu] & \bm m_n (z_1) & \cdots & \bm m_n (z_m)\end{pmatrix},
\]
using \eqref{app key eqn} and the first equality of \eqref{app bordered hankel M}, one can readily show that
\begin{equation}
\bm{M} = L_n \bm H_n^{B}[\mu;z_1,\cdots,z_m]\
\begin{pmatrix}
L_n^{[n-m]} & \bm 0\\
\bm 0 & I_m
\end{pmatrix}^\top.
\end{equation}
Combining it with the second equality of \eqref{app bordered hankel M}, we obtain the desired identity,
\[
\det\bm H_n^{B}[\mu;z_1,\cdots,z_m]= \det \bm{H}_{n-m}[\mu] \det \bm M_{22}.
\]
\end{proof}

%%%%%%%%%%%%%%%%%%%%%%%%%%%%%%%%%%%%%%%%

\subsection{Striped Hankel Determinantal Formula for type II MOPs}

In this subsection, we will show that our Theorem~\ref{typeIIreduction} gives the following standard expression, see for example \cite[Equation (2.7)]{Kui10}, of the type II MOPs with the striped Hankel determinants (and vice versa),
\begin{align} \label{striped hankel det formula}
P_{\vec{n}}(x)=	\frac{\det \bm{H}^S_{n+1}[\vec{w}; x]}{\det \bm{H}^{S}_{n}[\vec{w}]}
\end{align}
where
\[
\bm{H}^{S}_{n+1}[\vec{w}; x] = \left( \begin{smallmatrix}
w_0^{(1)} & w_1^{(1)} & \cdots & w_{n_1-1}^{(1)} & \blue{w_0^{(2)}} & \cdots & \blue{w_{n_2-1}^{(2)}} & \boldsymbol \cdots & \violet{w_0^{(r)}} & \cdots & \violet{w_{n_r-1}^{(r)}} & 1 \\[6pt]
w_1^{(1)} & w_2^{(1)} & \cdots & w_{n_1}^{(1)} & \blue{w_1^{(2)}} & \cdots & \blue{w_{n_2}^{(2)}} & \boldsymbol \cdots & \violet{w_1^{(r)}} & \cdots & \violet{w_{n_r}^{(r)}} & x \\[6pt]
\vdots & \vdots & \reflectbox{$\ddots$} & \vdots & \vdots & \reflectbox{$\ddots$} & \vdots & \boldsymbol \cdots & \vdots & \reflectbox{$\ddots$} & \vdots & \vdots \\[6pt]
w_{n-1}^{(1)} & w_{n}^{(1)} & \cdots & w_{n_1+n-2}^{(1)} & \blue{w_{n-1}^{(2)}} & \cdots & \blue{w_{n_2+n-2}^{(2)}} & \boldsymbol \cdots & \violet{w_{n-1}^{(r)}} & \cdots & \violet{w_{n_r+n-2}^{(r)}} & x^{n-1} \\[6pt]
w_{n}^{(1)} & w_{n+1}^{(1)} & \cdots & w_{n_1+n-1}^{(1)} &  \blue{w_{n}^{(2)}} & \cdots & \blue{w_{n_2 + n-1}^{(2)}} & \boldsymbol \cdots & \violet{w_{n}^{(r)}} & \cdots & \violet{w_{n_r+n-1}^{(r)}} & x^{n} 
\end{smallmatrix} \right)
\]
is an $(n+1) \times (n+1)$ striped Hankel matrix with variable $x$ and
\[
\bm{H}^{S}_{n}[\vec{w}] = 
\left(
\begin{smallmatrix}
w_0^{(1)} & w_1^{(1)} & \cdots & w_{n_1-1}^{(1)} & \blue{w_0^{(2)}} & \cdots & \blue{w_{n_2-1}^{(2)}} & \boldsymbol\cdots & \violet{w_0^{(r)}} & \cdots & \violet{w_{n_r-1}^{(r)}}  \\[6pt]
w_1^{(1)} & w_2^{(1)} & \cdots & w_{n_1}^{(1)} & \blue{w_1^{(2)}} & \cdots & \blue{w_{n_2}^{(2)}} & \boldsymbol\cdots & \violet{w_1^{(r)}} & \cdots & \violet{w_{n_r}^{(r)}} \\[6pt]
\vdots & \vdots & \reflectbox{$\ddots$} & \vdots & \vdots & \reflectbox{$\ddots$} & \vdots & \boldsymbol \cdots & \vdots & \reflectbox{$\ddots$} & \vdots  \\[6pt]
w_{n-1}^{(1)} & w_{n}^{(1)} & \cdots & w_{n_1+n-2}^{(1)} & \blue{w_{n-1}^{(2)}} & \cdots & \blue{w_{n_2 + n-2}^{(2)}} & \cdots & \violet{w_{n-1}^{(r)}} & \cdots & \violet{w_{n_r+n-2}^{(r)}}
\end{smallmatrix}
\right)
\]
is an $n \times n$ striped Hankel matrix such that each $w^{(j)}_k$ is a moment of the measure $w^{(j)}(z) \dd z$,
\[
w^{(j)}_k = \int_{\R} z^k w^{(j)}(z) \dd z.
\]

\begin{lemma}
Let $L_{n+1}$ be a unit lower-triangular $(n+1) \times (n+1)$ matrix such that \(\bm p_{n+1}(x) = L_{n+1} \bm m_{n+1}(x)\), where 
\begin{align}
&\bm p_{n+1}(x) = \begin{pmatrix}
    \bm p_{1}(x)\\
    \bm p_{2}(x)\\
    p_n(x)
\end{pmatrix}, \quad 
\bm m_{n+1}(x) = \left( 1,x,\cdots,x^{n} \right)^\top,
\end{align}
and $\bm p_{1}(x)$, $\bm p_{2}(x)$ are defined as in \eqref{defn of pi}, \eqref{defn of pi phi v, 1-2}. 
For $j = 1, \cdots, r$, we denote the $n_j \times n_j$ upper left block of $L_{n+1}$ by $L_{n+1}^{[n_j]}$, which satisfies
\[
\begin{pmatrix}
    p_0(x)\\
    p_1(x)\\
    \vdots\\
    p_{n_j-1}(x)
\end{pmatrix} = L_{n+1}^{[n_j]} \begin{pmatrix}
    1\\
    x\\
    \vdots\\
    x^{n_j-1}
\end{pmatrix}.
\]
Then, it holds that
\begin{align}
    L_{n+1} \bm{H}^{S}_{n+1}[\vec{w}; x] \begin{pmatrix}
    \begin{array}{c|c|c|c|c}
      L_{n+1}^{[n_1]} & \bm{0} & \cdots & \bm 0 & \bm 0\\
      \noalign{\vskip 2pt\hrule\vskip 2pt}
      \bm 0 & L_{n+1}^{[n_2]} & \ddots & \vdots & \vdots \\
      \noalign{\vskip 2pt\hrule\vskip 2pt}
      \vdots & \ddots & \ddots & \bm 0 & \vdots \\
      \noalign{\vskip 2pt\hrule\vskip 2pt}
      \bm 0 & \cdots & \bm 0 & L_{n+1}^{[n_r]} & \bm 0\\
      \noalign{\vskip 2pt\hrule\vskip 2pt}
        \bm 0 & \cdots & \cdots & \bm 0 & 1
    \end{array}
    \end{pmatrix}^\top = \begin{pmatrix}
    \begin{array}{c|c|c}
        \bm{H}_1 & \bm{M}_{12} & \bm{p}_1(x) \\
        \noalign{\vskip 2pt\hrule\vskip 2pt}
        \bm{0} & \bm{M}_{22} & \bm{p}_2(x) \\
        \noalign{\vskip 2pt\hrule\vskip 2pt}
        \bm{0} & \bm{v}_2^\top & p_n(x)
    \end{array}
\end{pmatrix},
\end{align}
where $\bm{H}_1$, $\bm{M}_{12}$, $\bm{M}_{22}$, $\bm{v}_2$ are defined as in \eqref{defn of H}, \eqref{defn of M}, \eqref{defn of M, H 4 blocks}, \eqref{defn of v}, \eqref{defn of pi phi v, 1-2}.
Consequently, we have
\begin{align} \label{det form 1}
\det \bm{H}^{S}_{n+1}[\vec{w}; x] = 
\det \bm{H}_1 \det \begin{pmatrix}
\bm{M}_{22}  & \bm{p}_2(x)  \\
\bm{v}_2^\top & p_n(x)
\end{pmatrix}.
\end{align}
\end{lemma}

\begin{proof}
We omit the proof as it is similar to the one for Proposition~\ref{app bordered hankel prop}.
\end{proof}

\begin{lemma}
Let $L_{n}$ be the unit lower–triangular matrix such that $\bm p_n (x) = L_{n} \bm m_{n} (x)$,
where 
\begin{align}
&\bm p_{n}(x) = \begin{pmatrix}
    \bm p_1(x)\\
    \bm p_2(x)
\end{pmatrix}, \quad 
\bm m_{n}(x) = \left( 1,x,\cdots,x^{n-1} \right)^\top.
\end{align}
Again, we denote the $n_j \times n_j$ upper left block of $L_n$ by $L_n^{[n_j]}$ for $j = 1, 2, \cdots, r$. Then, it holds that
\begin{equation}\label{eq:singleOP-factorization}
\bm{M} = L_{n} \, \bm{H}^{S}_{n}[\vec{w}] 
\begin{pmatrix} \begin{array}{c|c|c|c}
      L_{n}^{[n_1]} & \bm{0} & \cdots & \bm 0 \\
      \noalign{\vskip 2pt\hrule\vskip 2pt}
      \bm 0 & L_{n}^{[n_2]} & \ddots & \vdots \\
      \noalign{\vskip 2pt\hrule\vskip 2pt}
      \vdots & \ddots & \ddots & \bm 0 \\
      \noalign{\vskip 2pt\hrule\vskip 2pt}
      \bm 0 & \cdots & \bm 0 & L_{n}^{[n_r]}
    \end{array}
    \end{pmatrix}^\top,
\end{equation}
where $\bm{M}$ is defined as in \eqref{defn of M}, and
\begin{align} \label{det form 2}
\det\bm{M} = \det \bm{H}^{S}_{n}[\vec{w}].
\end{align}
\end{lemma}

\begin{proof}
Again, we omit the proof, as it is similar to the one for Proposition~\ref{app bordered hankel prop}.
\end{proof}

By \eqref{M block structure}, one can readily show that
\begin{align} \label{det form 3}
\det \bm{M} = \det \bm{H}_1 \det \bm{M}_{22}
\end{align}
By \eqref{det form 1}, \eqref{det form 2}, and \eqref{det form 3}, we have
\[
\frac{\det \bm{H}^{S}_{n+1}[\vec{w}; x]}{\det \bm{H}^{S}_{n}[\vec{w}]} = \frac{1}{\det \bm{M}_{22}} \det
\begin{pmatrix}
\bm{M}_{22}  & \bm{p}_2(x)  \\
\bm{v}_2^\top & p_n(x)
\end{pmatrix}.
\]

%%%%%%%%%%%%%%%%%%%%%%%%%%%%%%%%%%%%%%%%%%%%%

\bibliographystyle{alpha}
\bibliography{references_new.bib}

\end{document}